\documentclass[12pt]{amsart}
\usepackage{amssymb}   
\usepackage{amsmath}  
\usepackage[hidelinks]{hyperref}  
\usepackage{enumitem}
\usepackage{xcolor}
\usepackage{tikz-cd}
\usepackage[normalem]{ulem}  
\usepackage{mathtools} 
\usepackage{etoolbox}
\usepackage{varioref}
\hypersetup{
    colorlinks,
    linkcolor={red!50!black},
    citecolor={blue!50!black},
    urlcolor={blue!80!black}
}

\newtheorem{theorem}{Theorem}
\newtheorem{lemma}[theorem]{Lemma}

\newtheorem{definition}[theorem]{Definition}
\newtheorem{corollary}[theorem]{Corollary}
\newtheorem{conjecture}[theorem]{Conjecture}

\theoremstyle{remark}
\newtheorem{remark}[theorem]{Remark}

\newcommand{\F}{\mathbb{F}}
\newcommand{\Z}{\mathbb{Z}}
\newcommand{\Q}{\mathbb{Q}}
\newcommand{\R}{\mathbb{R}}
\newcommand{\C}{\mathbb{C}}

\newcommand{\Cc}{\mathbb{C}^{\times}}

\newcommand{\cA}{\mathcal{A}}
\newcommand{\cB}{\mathcal{B}}

\newcommand{\cH}{\mathcal{H}}

\newcommand{\cO}{\mathcal{O}}

\newcommand{\til}{\tilde}

\newcommand{\ra}{\rightarrow}
\newcommand{\ira}{\hookrightarrow}
\newcommand{\sra}{\twoheadrightarrow}
\newcommand{\xra}{\xrightarrow}

\newcommand{\Lie}{\operatorname{Lie}}

\newcommand{\depth}{\operatorname{depth}}
\newcommand{\Ad}{\operatorname{Ad}}

\newcommand{\Irr}{\operatorname{Irr}}
\newcommand{\Ind}{\operatorname{Ind}}

\newcommand{\Res}{\operatorname{Res}}
\newcommand{\Tr}{\operatorname{Tr}}
\newcommand{\Nm}{\operatorname{Nm}}
\newcommand{\Rep}{\operatorname{Rep}}
\newcommand{\Gal}{\operatorname{Gal}}
\newcommand{\val}{\operatorname{val}}

\newcommand{\fg}{\mathfrak{g}}
\newcommand{\fh}{\mathfrak{h}}
\newcommand{\ft}{\mathfrak{t}}

\newcommand{\fm}{\mathfrak{m}}

\newcommand{\fa}{\mathfrak{a}}

\newcommand{\sS}{\mathsf{S}}
\newcommand{\sG}{\mathsf{G}}
\newcommand{\sP}{\mathsf{P}}
\newcommand{\sL}{\mathsf{L}}

\newcommand{\NN}{\mathcal{N}}

\newcommand{\matr}[1]{\left[\begin{matrix}#1\end{matrix}\right]}

\newcommand{\depthzerofirstentry}[1]{#1}

\newcommand{\dsp}{\displaystyle}

\newcommand{\Hom}{\operatorname{Hom}}

\newcommand{\Stab}{\operatorname{Stab}}
\newcommand{\ind}{\operatorname{ind}}
\newcommand{\im}{\operatorname{im}}

\newcommand{\SL}{\mathrm{SL}}
\newcommand{\GSP}{\mathrm{GSp}}
\newcommand{\SU}{\mathrm{SU}}
\newcommand{\DLr}{\operatorname{DL}_r}
\newcommand{\DLzero}{\operatorname{DL}_0}
\newcommand{\DLpi}{\operatorname{DL}_{\rho(\pi)}}

\newcommand{\quo}{/\hspace{-0.3em}/}

\newcommand{\MP}{\operatorname{MP}}  
\newcommand{\TPzero}{\operatorname{RP}_0} 
\newcommand{\TPr}{\operatorname{RP}_r} 
\DeclareMathOperator{\Int}{Int}
\DeclareMathOperator{\Group}{I}
\DeclareMathOperator{\changeFr}{\mathcal{F}}
\DeclareMathOperator{\res}{res}

\begin{document}

\title{Langlands parameters for Moy-Prasad types}

\author{Tsao-Hsien Chen}
\address{School of Mathematics, University of Minnesota, Vincent Hall, Minneapolis, MN-55455, USA}
\email{chenth@umn.edu}

\author{Stephen DeBacker}
\address{University of Michigan\\
Ann Arbor, MI 48109-1043, USA}
\email{stephendebacker@umich.edu}

\author{Cheng-Chiang Tsai}
\address{Institute of Mathematics, Academia Sinica, 6F, Astronomy-Mathematics Building, No. 1,
Sec. 4, Roosevelt Road, Taipei, Taiwan \vskip.2cm
also Department of Applied Mathematics, National Sun Yat-Sen University, and Department of Mathematics, National Taiwan University}

\email{chchtsai@gate.sinica.edu.tw}

\begin{abstract} 
Suppose $G$ is a tamely ramified $p$-adic reductive group.  
We construct
a partial local Langlands correspondence between  the set of 
irreducible smooth  representations of $G$  having depth $r$
and a certain set of $G^\vee$-conjugacy classes of continuous homomorphisms $\varphi:\Group_F^r\ra G^{\vee}$. Here   $G^\vee$ is the dual group of $G$, and  $\Group_F^r$ is the $r^{\text{th}}$ upper-numbering filtration subgroup of the inertia subgroup $\Group_F$.
\end{abstract}

\makeatletter
\let\@wraptoccontribs\wraptoccontribs
\makeatother


  \date{\today}

\maketitle

\setcounter{tocdepth}{1}
\tableofcontents

\section*{Introduction}  
Suppose $F$ is a non-archimedean local field and $G$ is a connected reductive group that splits over a tame extension of $F$.
 Let $G^\vee$ denote the  dual group of $G$ and for $s \geq 0$ let $\Group_F^s$ denote the $s^{\text{th}}$ upper numbering filtration subgroup of the inertia subgroup $\Group_F$.

Suppose $(\pi,V)$ is an irreducible smooth complex representation of $G(F)$.  
Since $G$ splits over a tame extension,  the existence of  Moy-Prasad isomorphisms (see the discussion in Section~\ref{subsub:mpiso}) combined with the argument of~\cite[Thm. 5.2]{MP94} shows that either (i) $\pi$ has nontrivial fixed vectors under the pro-unipotent radical of some parahoric subgroup of $G(F)$ or (ii) there exists $r \in \Q_{>0}$ such that $\pi$ has a nondegenerate Moy-Prasad type of depth $r$.    In the former case we say that $\pi$ has depth zero and in the latter case we say that $\pi$ has positive depth $r$.  In either case we denote the depth of $\pi$ by $\rho(\pi)$.

 In this paper we construct a continuous homomorphism $\varphi_{\pi}  \colon \Group_F^{\rho(\pi)} \ra G^\vee$, which, up to $G^\vee$-conjugation,   depends only on $\pi$. (If $\rho(\pi)=0$ and $G$ does not split over an unramified extension, then see Definition \ref{def:toralzeromapa} for the precise object we construct.) We establish properties of the assignment $\pi \mapsto \varphi_\pi$ including the fact that the depth of $\varphi_\pi$ is less than or equal to $\rho(\pi)$.  Moreover, if $\rho(\pi) \in \Z_{(p)}:=\Q\cap\Z_p$, then the depth of $\varphi_\pi$ is $\rho(\pi)$.

\subsection*{A more detailed description}  
Let $\cO_F$ denote the ring of integers of $F$, $\fm_F$ the maximal ideal in $\cO_F$, and $k_F:=\cO_F/\fm_F$ the residue field. Denote by $\bar{F}$ a separable closure of $F$.  Whenever we say that ``$E/F$ is an extension''
we will mean that $E$ is a subextension of  $\bar{F}/F$. Let $p:=\operatorname{char}(k_F)$. We denote by $\val:\bar{F}\ra\Q\sqcup\{+\infty\}$ the valuation map, normalized so that $\val(F^{\times})=\Z$.
Fix  a nontrivial additive character $\Lambda = \Lambda_F:F\ra\Cc$.  We assume $\res_{\fm_F}\!\Lambda$, the restriction of $\Lambda$ to $\fm_F$, is trivial, but  $\res_{\cO_F}\!\Lambda$ is not trivial.  In this way we may also regard $\Lambda$ as a nontrivial character on $k_F$.

 For any extension $E/F$ we denote by $\cB(G,E)$ the reduced Bruhat-Tits building for $G$ over $E$. In Sections~\ref{sec:moy-prasadrecollections} and~\ref{sec:duallattice} we discuss some properties of the building, our notational conventions for the Moy-Prasad filtration lattices and subgroups, the properties of these lattices and subgroups with respect to base change, and our conventions for dual Lie algebras.

\subsubsection*{Moy-Prasad types}
For any $r\in\Q_{>0}$,
a \textbf{Moy-Prasad type of depth} $r$ \cite[\S5.1 \& \S3.8]{MP94} for $G(F)$ is
a pair $(x,X)$ where $x\in\cB(G,F)$ and $X\in\fg^*(F)_{x=-r}$. Here $\fg^*(F)_{x=-r}$ is the Moy-Prasad (sub)quotient of the dual Lie algebra.  We let $\MP_r$ denote the set of  Moy-Prasad types for $G(F)$ of depth $r$.   A pair $(x,X) \in \MP_r$ 
is said to be {\bf nondegenerate} provided that  $ X+\fg^*(F)_{x>-r}$ does not contain an element of the nilpotent cone $\NN^*$ (see Section~\ref{subsub:whereNis} for the definition of $\NN^*$).   

For any $(x,X) \in \MP_r$ consider the $1$-dimensional representation
\begin{equation}\label{eq:MPtypeongroup}
\rho_X:G(F)_{x=r}\ra\Cc,\;\;\rho_X(W)=\Lambda_F(\log(W),X)
\end{equation}
where $\log:G(F)_{x=r}\xra{\sim}\fg(F)_{x=r}$ is the Moy-Prasad isomorphism; see Section~\ref{subsub:mpiso}. We may and do view $\rho_X$ as a $1$-dimensional representation of $G(F)_{x\ge r}$ via the quotient map $G(F)_{x\ge r}\sra G(F)_{x=r}$.

 A \textbf{depth zero Moy-Prasad type} 
    is a pair $(\depthzerofirstentry{x},\chi)$ where $x\in \cB(G,F)$ and  $\chi$ is a cuspidal representation of $G(F)_{x=0}$
inflated to the parahoric subgroup $G(F)_{x\geq0}$.

An irreducible smooth representation $\pi$ of $G(F)$ is said to have $(x,X)$ (resp. $(\depthzerofirstentry{x}, \chi)$) as a Moy-Prasad type of depth $r >0$ (resp. $r = 0$) provided that  $\res_{G(F)_{x\ge r}}\!\pi$,  the restriction of $\pi$ to $G(F)_{x\ge r}$,  contains $\rho_X$ (resp. $\chi$).

\subsubsection*{Restricted  parameters}   
For any connected reductive group $H$ over any field (typically $F$), we denote by $H^{\vee}$ the (Langlands) dual connected reductive group over $\C$. 

We let $(\Group_F^r \, | \, r \geq 0)$ denote the upper numbering filtration of the Weil group $W_F$ of $F$ (see~\cite[Section 9 of Chapter 1]{cassels-frohlich} or~\cite[Section 3 of Chapter IV]{serre:local} or~Section~\ref{sec:nlcft}).  We  define $\Group_F^{r+}$ to be the closure of $\bigcup_{s>r}\Group_F^s$, so $\Group_F^{r+} = \displaystyle \lim_{\leftarrow} \Group(E/F)^{r+}$, where $\Group(E/F)^{r^+} = \bigcup_{s >r} \Group(E/F)^s$.
Note that $\Group_F^0$ is the inertia subgroup of $W_F$, and $\Group_F^{0+}$ is the wild inertia subgroup of $W_F$.   The group $\Group_F^r$ carries the subspace topology induced from the topology on $\Group_F$, and we equip $\Group_F^{r:r+}:= \Group_F^r/\Group_F^{r+}$ with the quotient topology.

Suppose $r \in \Q_{>0}$.  A continuous homomorphism $\varphi:\Group_F^{r:r+}\ra G^{\vee}$ or, alternatively, a  continuous homomorphism $\varphi:\Group_F^r\ra G^{\vee}$ that is trivial on $\Group_F^{r+}$ is a {\bf tame restricted depth-$r$ parameter} provided that  there exist a maximal torus $T^{\vee}\subset G^{\vee}$ and a continuous homomorphism $\til{\varphi}:\Group_F^{0+}\ra T^{\vee}$, trivial on $\Group_F^{r+}$, such that $\til{\varphi}|_{\Group_F^r}\equiv\varphi$. 
We let $\TPr$ denote the set of 
$G^{\vee}$-conjugacy classes
of tame restricted depth-$r$ parameters.

 A \textbf{restricted depth zero parameter} is a continuous cocycle
    \[(\varphi:\Group_F^{0:0+}\to G^\vee)\in Z^1(\Group_F^{0:0+},G^\vee)\]
    that admits an extension to a depth zero 
    Langlands parameter (see Definition~\ref{def:toralzeromapa}).
    We write 
    $\TPzero$ for the set of 
    $G^\vee$-conjugacy classes of 
    restricted depth zero parameters.

\subsubsection*{Statement of the main result}
    For $s \in \Q_{\geq 0}$ we let $\Irr(G(F))_s$ denote the set of (equivalence classes of) irreducible representations of $G(F)$ of depth $s$.
Most of this article is concerned with constructing
and establishing properties of the following map:

\begin{theorem} \label{thm:mainthm} For $r \in \Q_{\geq 0}$ such that $\Irr(G(F))_r \neq \emptyset$  there is a map
\[
\Irr(G(F))_r \ra \TPr.
\] 
Moreover, if we also have $r \in \Z_{(p)} \cap \Q_{>0}$, then every element of the image of this map is nontrivial. 
\end{theorem}

Suppose $\pi \in \Irr(G(F))_r$. When $r$ is positive, the proof has two steps.  Suppose $(x,X)$ is a nondegenerate Moy-Prasad type for $\pi$. In Section~\ref{sec:MPtypetoDLparam}  we associate to the pair $(x,X)$ an equivalence class of pairs $(\alpha, Y)$, called a {\bf depth-$r$ Deligne-Lusztig parameter},
where $\alpha \in \bar{F}$ such that $\val(\alpha) = r$ and $Y$ is a measure of the semisimple part of $X$ ``transported by $\alpha$'' to $\fg^*(\bar{F})_{x=0}$ (see Definition~\ref{def:DL}).  Then, in Section~\ref{sec:C2Bnew} we construct from the  pair $(\alpha,Y)$   a $G^\vee$-conjugacy class of restricted depth-$r$ parameters.   The resulting map from $\Irr(G(F))_r$ to $\TPr$   is independent of all choices.     The structure of the construction in the depth zero case  is very similar.
See Lemmas~\ref{lem:pfinpos} and~\ref{lem:pfinzero} for the proof of Theorem~\ref{thm:mainthm}.

\subsection*{Structure of this paper}  In Section~\ref{sec:nlcft} we introduce a normalized Hasse-Herbrand function and establish some facts about ramification groups that are needed in the sequel, especially when $r\not\in\Z_{(p)}$.
In Sections~\ref{sec:constructiontametori} through~\ref{sec:conjnew} we prove and discuss Theorem~\ref{thm:mainthm} for positive depth representations.
In Section~\ref{sec:constructiontametori}  we provide a preliminary  construction of the map of Theorem~\ref{thm:mainthm} in the case when $G$ is a tamely ramified
torus. In Section~\ref{sec:MPtypetoDLparam} we attach to any associativity class of depth-$r$ Moy-Prasad types  a  depth-$r$ Deligne-Lusztig parameter. We then construct  in Section~\ref{sec:C2Bnew} a bijection between the set of depth-$r$ Deligne-Lusztig parameters and $ \TPr$, proving Theorem~\ref{thm:mainthm}.
In Section~\ref{sec:conjnew} we conjecture that the (usual) Local Langlands Correspondence is an extension of our construction, and prove this conjecture for the Local Langlands Correspondence constructed by Kaletha \cite{Kal19,Kal19b} for semisimple supercuspidal representations.  In Section~\ref{sec:depthzero} we, broadly speaking, follow the outline sketched above for positive depth representations to prove Theorem~\ref{thm:mainthm} for depth zero representations.

\section{Notation and conventions}\label{sec:notation}
Recall that $G$ is a connected reductive $F$-group that splits over a tame extension of $F$.

 Let $F^t$ be the maximal tame extension of $F$ in $\bar{F}$, and denote by $F^u$ the maximal unramified extension of $F$ in $F^t$. Denote by $\bar{k}:=k_{F^u}=k_{F^t}= k_{\bar{F}}$ their algebraically closed residue field.    Whenever we say that ``$E/F$ is a tame extension''
we will mean that $E$ is a subextension of  $F^t/F$. For any  extension $E/F$ we have the corresponding objects $\cO_E, \fm_E$, $k_E$, $E^u$, $E^t$, etc.

We fix a maximal $F$-torus $A$ in $G$ for which
(a) $A^u$, the maximal $F^u$-split subtorus  of $A$, is a maximal $F^u$-split $F$-torus in $G$ and (b)
 $A^F$, the maximal $F$-split subtorus of $A$, is a maximal $F$-split torus in $G$.
Since $G$ is quasi-split over $F^{u}$, the centralizer of $A^u$ in $G$ is $A$; that is
$A=Z_G(A^u)$.   It follows from~\cite[Theorem~6.1]{prasad:unramified} (or from~\cite[Theorem~3.4.1]{debacker:maximalunramified}  by looking at pairs of the form $(F,{\sf T}) \in I^m$  with $F$ an alcove) that $A$
exists and is unique up to  $G(F)$-conjugacy.  

\subsection{Buildings, apartments, Moy-Prasad filtrations}  \label{sec:moy-prasadrecollections}  Because we are going to need to work over finite extensions of $F$, we present this material from the point of view of an extension $E$ of $F$.

 For any  extension $E/F$ of finite ramification index we let $\cB(G,E)$ denote the  reduced Bruhat-Tits building of $G(E)$.   For any finite extension $E'/E$ we have $\cB(G,E) \subset \cB(G,E')$, and if $E'/E$ is Galois and tame we have a canonical bijection $\cB(G,E)=\cB(G,E')^{\Gal(E'/E)}$ (\cite{Rou77} or \cite{Pra01}). 
This allows us to conveniently define $\cB(G,E^t)$ to be the direct limit (=union) of all $\cB(G,E')$ for $E'/E$ finite, Galois, and tame (by convention we assume $E' \leq E^t$).

 We let $\cA$ denote the apartment of $A^F$  in $\cB(G,F) \subset \cB(G,E)$, and we let $\tilde{\cA}$ denote the apartment of $A^u$ in $\cB(G,F^u) \subset \cB(G,E^u)$.    Any apartment in $\cB(G,F)$ is $G(F)$-conjugate to $\cA$.  

 \subsubsection{Moy-Prasad filtrations} 
For any $x\in\cB(G,E)$ and  $r\in\R_{\ge 0}$ we have
the Moy-Prasad filtration subgroup $G(E)_{x\ge r}\subset G(E)$;  in \cite[\S2.6]{MP94} the subgroup $G(F)_{x \ge r}$ is  denoted $P_{x,r}$.  We have that $G(E)_{x\ge 0}$ is the parahoric subgroup attached to $x$ and 
for any $s \geq r \geq 0$
we have that $G(E)_{x\ge s}$
is an open normal subgroup of $G(E)_{x \ge r}$.
 We define $G(E)_{x>r}:=\bigcup_{s>r} G(E)_{x\ge s}$, and $G(E)_{x=r}:=G(E)_{x\ge r}/G(E)_{x>r}$. Likewise, for any $r\in\R$ there is the Moy-Prasad filtration sub-lattice $\fg(E)_{x\ge r}\subset\fg(E)$ as well as $\fg(E)_{x>r}\subset\fg(E)$ and the quotient $\fg(E)_{x=r}:=\fg(E)_{x\ge r}/\fg(E)_{x>r}$; 
 in \cite[\S3.2]{MP94} the lattice $\fg(F)_{x \ge r}$ is denoted $\fg_{x,r}$. 
 Note that $\fg(E)_{x\ge r}$ is an $\cO_E$-module and $\fg(E)_{x=r}$ is a $k_E$-vector space. 
 
\subsubsection{Moy-Prasad isomorphisms}  \label{subsub:mpiso}
 Since $G$ splits over $F^t \leq E^t$, the torus $A$ is a weakly induced torus~\cite[Definitions~B.6.2 and~2.5.1]{KP23}.   Thanks to~\cite[Theorem~13.5.1]{KP23} and its proof, 
 we have a Moy-Prasad isomorphism 
\[\log_E \colon G(E^u)_{x=r}\cong\fg(E^u)_{x=r}\]
 for all 
$x \in \cA$ and $r \in \R_{>0}$, hence for all 
$x \in \cB(G,F)$ and $r \in \R_{>0}$.   This isomorphism is $\Gal(E^u/E)$-equivariant.  Indeed, since $x \in \cB(G,F)$, from~\cite[Theorem~13.5.1(2)]{KP23}   it is compatible with $\log = \log_F \colon G(F^u)_{x=r}\cong\fg(F^u)_{x=r}$.   Note that in \cite[Section 13.5]{KP23} tori are assumed to carry the minimal congruence filtration of~\cite[Section 5]{Yu15}.  However, for weakly induced tori we know~\cite[Corollary~B.10.13]{KP23} that the standard filtration~\cite[Definition~B.5.1]{KP23} of $A(E^u)$ agrees with the minimal congruence filtration of $A(E^u)$.  We remark that the standard filtration is what is used in~\cite{MP94} (where they denote the torus $A$ as $Z$).

\subsection{The dual of the Lie algebra}  \label{sec:duallattice}
We put $\fg^*(F) :=\Hom_F(\fg(F),F)$ the dual vector space; it has the functoriality that $\fg^*(E):=\Hom_E(\fg(E),E)$ is canonically identified with $\fg^*(F)\otimes_FE$ where $E$ is a finite  extension of $F$.  In particular we have a natural embedding $\fg^*(F)\ira\fg^*(E)$. If $T\subset G$ is a maximal torus defined over $E$, then 
$\fg^*=\ft^*\oplus(\ft^*)^{\perp}$ where $(\ft^*)^{\perp}$ is the subspace of $\fg^*$ spanned by the root spaces in $\fg^*$ for the roots of $G$ with respect to $T$.  In particular we will identify $\ft^*$ with the $T$-fixed subspace of $\fg^*$. 
We define Moy-Prasad filtration lattices and quotients
\[
\begin{array}{l}
\fg^*(E)_{x\ge -r}:=\{\lambda \in\fg^*(E)\;|\;\lambda (X)\in\fm_E,\;\forall X\in\fg(E)_{x>r}\}\\
\fg^*(E)_{x>-r}:=\bigcup_{s>-r}\fg^*(E)_{x\ge s}=\{\lambda\in\fg^*(E)\;|\;\lambda(X)\in\fm_E,\;\forall X\in\fg(E)_{x\ge r}\}\\
\fg^*(E)_{x=-r}=\fg^*(E)_{x\ge -r}/\fg^*(E)_{x>-r}.
\end{array}
\]
By definition we have a perfect pairing
\[
\fg(E)_{x=r}\times\fg^*(E)_{x=-r}\ra k_E.
\]

For the remainder of this subsection, we assume that $E'/E$ is tame or that $G$ is $E$-split.  This allows us to use~\cite[Lemma 2.3.5(3)]{Spi21}.  We have
\[G(E)_{x\ge r}=G(E)\cap G(E')_{x\ge r} \text{ (for }r>0\text{)}\text{ and } \fg(E)_{x\ge r}=\fg(E)\cap \fg(E')_{x\ge r}.\]

The above identities induce canonical embeddings $G(E)_{x=r}\ira G(E')_{x=r}$ and $\fg(E)_{x=r}\ira\fg(E')_{x=r}$. Therefore for any $E'/E$ tame (and possibly infinite) we can  define $G(E')_{x\ge r}$, $\fg(E')_{x\ge r}$, $G(E')_{x=r}$, and $\fg(E')_{x\ge r}$ in terms of direct limits.    We  also have 
\[
\fg^*(E)_{x\ge -r}=\fg^*(E)\cap \fg^*(E')_{x\ge -r} \text{ and }\fg^*(E)_{x=-r}\ira\fg^*(E')_{x=-r}.
\]
Moreover, when $E'/E$ is also Galois, we have a surjection $\Gal(E'/E)\sra\Gal(k_{E'}/k_E)$ and the action $\Gal(E'/E)\curvearrowright\fg(E')_{x\ge r}$ induces an action $\Gal(k_{E'}/k_E)\curvearrowright\fg(E')_{x=r}$ so that we have natural identifications
\[
G(E)_{x\ge r}=\left(G(E')_{x\ge r}\right)^{\Gal(E'/E)},\;\;\fg(E)_{x\ge r}=\left(\fg(E')_{x\ge r}\right)^{\Gal(E'/E)},\]
\[
\fg(E)_{x=r}=\left(\fg(E')_{x=r}\right)^{\Gal(k_{E'}/k_E)}, \text{ and }\fg^*(E)_{x=-r}=\left(\fg^*(E')_{x=-r}\right)^{\Gal(k_{E'}/k_E)}.
\]

When $G=T$ is a torus, $\cB(G,E)$ consists of only one point, call it $x$, and we suppress $T(E)_{x\ge r}$ to $T(E)_{\ge r}$, $\ft(E)_{x=r}$ to $\ft(E)_{=r}$, and so on.

Finally, if ${Y} \in \fg^*(E)_{x \geq r}$, then we let $\bar{Y}$ denote the image of $Y$ in $ \fg^*(E)_{x = r}$.   Similarly, if $\bar{X} \in \fg^*(E)_{x = r}$, then we let ${X} \in \fg^*(E)_{x \geq r}$ denote a lift of $\bar{X}$.  We use an analogous convention  for elements of $\fg(E)$ and $G(E)$.   We also often denote a lift of $U \in \fg^*(E)_{x = r}$ by $U_0 \in \fg^*(E)_{x \geq r}$.

\section{Trace, norm, and ramification groups} \label{sec:nlcft}

In the standard treatment of ramification  groups, the valuation is chosen so that the lower numbering subgroups are indexed by integers.  Here, we do everything relative to the unique valuation on $\bar{F}$ for which $\val(F^\times) = \Z$.    More details about this normalized approach to defining the upper and lower numbering subgroups of local class field theory may be found in the expository paper~\cite{DST25}.

If $K,L$ are finite extensions of $F$ with $K \leq L$, then $e(L/K)$ denotes the ramification degree of $L/K$.  If $L/K$ is Galois, then $\Gal(L/K)$ denotes the Galois group and $\Group(L/K)$ denotes the inertia subgroup.

Fix a finite extension  $E$ of $F$, and suppose $L/E$ is a finite Galois extension with uniformizer $\varpi_L$.
For $\sigma \in \Group(L/E)$
     set
\[\depth_{L/E}(\sigma) := \val \left( \dfrac{ \sigma(\varpi_L) - \varpi_L}{\varpi_L} \right).\]
This definition is independent of the choice of $\varpi_L$.  For  $r \in \R_{\geq 0}$  define the lower numbering ramification groups by
 \[\Group(L/E)_r := \{ \sigma \in \Group(L/E) \, : \, \depth_{L/E}(\sigma) \geq r \}. \]
For $i \geq 0$, the group $G_i$ of ~\cite[IV.2]{serre:local} is the group $\Group(L/E)_{i/e(L/F)}$.  We set $\ell(L/E) = \inf \{r \in \R_{\geq 0} \, | \, \Group(L/E)_r = 1 \}$.  This is  the last break for the lower numbering subgroups.

 We define normalized versions of the Hasse-Herbrand functions by 
\[\varphi_{L/E}(x) = \int_0^x \left| {\Group(L/E)_y} \right|\, dy = \sum_{\sigma \in \Group(L/E)} \min(\depth_{L/E} (\sigma), x) \]
and $\psi_{L/E}(z) = \varphi_{L/E}^{-1}(z)$ for $z \in \R_{\geq 0}$.  The standard, unnormalized version of these functions are given by $x \mapsto e(E/F) \cdot \varphi_{L/E}(x/e(L/F))$ and $y \mapsto e(L/F) \cdot \psi_{L/E}(y/e(E/F))$.  The normalized versions of these functions enjoy the usual composition rules: \(
\varphi_{L/E} = \varphi_{K/E} \circ \varphi_{L/K}\) and \(\psi_{L/E} = \psi_{L/K}\circ\psi_{K/E}\).  Note that the derivative of $\varphi_{L/E}$ is $1$ on $(\ell(L/E), \infty)$.  Thus, if  $L/E$ is tamely ramified, then we have $\ell(L/E) = 0$ and so $\psi_{L/E}(s) = \varphi_{L/E}(s) = s$ 
for all $s \in \R$.

The upper numbering subgroups are defined  by $\Group(L/E)^s := \Group(L/E)_{\psi_{L/E}(s)}$, and when $F = E$ this indexing agrees with the standard, unnormalized upper numbering of ramification groups. 
 We set $u(L/E) = \inf \{r \in \R_{\geq 0} \, | \, \Group(L/E)^r = 1 \}$.  This is  the last break for the upper numbering subgroups.  Note that $u(L/E) = \varphi_{L/E}(\ell(L/E))$.

If $K/E$ is a Galois extension and $K \leq L$, then for all $r, s \geq 0$ there are exact sequences
\begin{equation} \label{eq:exactnumberlower}
1 \ra \Group(L/K)_r \ra \Group(L/E)_r\ra \Group(K/E)_{\varphi_{L/K}(r)} \ra 1
\end{equation}
and
\begin{equation} \label{eq:exactnumberupper}
1 \ra   \Group(L/K)^{\psi_{K/E}(s)} \ra  \Group(L/E)^s \ra  \Group(K/E)^s \ra 1
\end{equation}

\subsection{A different different}  \label{sec:di9fferentdifferent}
For any finite separable extension $K/F$, we define $\Lambda_K:=\Lambda_F\circ\Tr_{K/F}:K\ra\Cc$. Since $\Tr_{K/F}:K\ra F$ is surjective, we have that $\Lambda_K$ is also a nontrivial continuous additive character. 
Define
\[c_{\Lambda_K}:=\max\{ \val(x)\,|\, x\in K,\;\Lambda_K(x)\not=1\}.
\]
Note that $c_{\Lambda_F} = 0$.

\begin{definition}   \label{defn:cdefined}
    Suppose $F \leq E \leq L$ is a tower of finite separable extensions.  Set
    \[c_{L/E}:=c_{\Lambda_E}-c_{\Lambda_L}\]
    and
    \[ \Lambda_{L/E} := \Lambda_E \circ \Tr_{L/E}.\]
\end{definition}

\begin{remark}
Note that $\Tr_{L/E}[L_{\geq s}] = E_{\geq s + c_{L/E}}$.
\end{remark}

\begin{remark} If $F \leq E \leq K \leq L$ is a tower of separable extensions, then
\( c_{L/E} = c_{L/K} + c_{K/E}\).
\end{remark}

\begin{lemma}\label{lem:cdifferentformula} Suppose $F \leq E \leq L$ is a tower of finite Galois extensions. We have
\[c_{L/E} = \sum_{1 \neq \sigma \in \Group(L/E)} \depth_{L/E}(\sigma).
\]
\end{lemma}

\begin{proof}
Define the normalized differental exponent by
 \[ d = d(L/E):= \max \{r \, | \, e(L/F)r \in \Z \text{ and } \Tr_{L/E}[L_{\geq -r}] \subset E_{\geq 0} \}.\]
From~\cite[VIII.1, Proposition~4]{We95} 
if $s$ satisfies
    \[ c_{\Lambda_E} -d - 1/e(L/F) < s \leq c_{\Lambda_E} -d + (e(L/E)-1)/e(L/F),\]
    then \(\Tr_{L/E}[L_s]= E_{c_{\Lambda_E}}\), 
    and if 
    \[t > c_{\Lambda_E}  -d + (e(L/E)-1)/e(L/F) = (c_{\Lambda_E} + 1/e(E/F)) -d - 1/e(L/F),\]
    then \(\Tr_{L/E}[L_t]\subset  E_{c_{\Lambda_E} + 1/e(E/F)} = E_{c_{\Lambda_E}+} \).
    Since $\Tr_{L/E}$ is $\cO_E$-linear,
   we conclude that $c_{\Lambda_L} = c_{\Lambda_E} - d + (e(L/E)- 1)/e(L/F)$.    Thus, 
 \[c_{L/E} = d(L/E)  + (1-e(L/E))/e(L/F).
\]
   According to~\cite[IV.2 Proposition~4]{serre:local}, this is
\(\sum_{1 \neq \sigma \in \Group(L/E)} \depth_{L/E}(\sigma) \).
\end{proof}

\begin{corollary}  \label{cor:cdifferentandHH}
  If $s \in \R_{\geq 0}$ satisfies $s \geq \ell(L/E)$, then $\varphi_{L/E}(s) = s + c_{L/E}$.    In particular, $c_{L/E} = u(L/E) - \ell(L/E)$.
\end{corollary}

\begin{proof}
Using the definition of $\varphi_{E/F}$, for $s \geq \ell(L/E)$  we have
\begin{equation*}
    \begin{split}
      \varphi_{L/E}(s) -s &=  - s + \sum_{\sigma \in \Group(L/E)} \min(\depth_{L/E} (\sigma), s)  
      =  \sum_{\sigma \in \Group(L/E) \setminus \{1\} } \depth_{L/E} (\sigma).  
    \end{split}
\end{equation*}
If $s = \ell(L/E)$, then $c_{L/E} = \varphi_{L/E}(\ell(L/E)) - \ell(L/E) = u(L/E) - \ell(L/E)$.
\end{proof}

\subsection{Some results about norm maps}

If $L/E$ is a finite Galois extension, then for all $s \in \R$ we have
\[ \Tr_{L/E}[L_{\geq s}] = E_{\geq (s + c_{L/E})} .\]
Since this is true for all $s$, for all $t \in \R$ we have a surjection, also denoted $\Tr_{L/E}$,
\[  \Tr_{L/E} \colon L_{= t} \sra E_{=(t + c_{L/E})}. \]

There are similar, but more subtle, results for the norm map $\Nm_{L/E} \colon L^\times \ra E^\times$.  For example, thanks to~\cite[V.2~Proposition 2 and V.6~Proposition 8]{serre:local}  for all $s \in \R_{>0}$ we have
\[ \Nm_{L/E}[L^{\times}_{\geq s}] =\Nm_{\tilde{L}/E}[\Nm_{L/\tilde{L}}[L^{\times}_{\geq s}]] \subset \Nm_{\tilde{L}/E}[\tilde{L}^{\times}_{\geq \varphi_{L/\tilde{L}}(s)} ] = E^{\times}_{\geq \varphi_{L/\tilde{L}}(s)}\]
where $\tilde{L}/E$ denotes the maximal unramified subextension of $E$ in $L$. 
Since $\varphi_{\tilde{L}/E}(t) = t$ for all $t \geq 0$, we conclude that $\Nm_{L/E}[L^\times_{\geq s}] \subset E^\times_{\geq \varphi_{L/E}(s)}$.  Since this is true for all $s \in \R_{>0}$,  we have a homomorphism, also denoted $\Nm_{L/E}$,
\[  \Nm_{L/E} \colon L^\times_{= s} \longrightarrow E^\times_{=\varphi_{L/E}(s)}. \]

\begin{lemma} \label{lem:normrange}
   Suppose $s \geq 0$.
   \begin{enumerate}
       \item \label{it:norm1} The homomorphism 
    \[  \Nm_{L/E} \colon L^\times_{\geq s} \longrightarrow E^\times_{\geq \varphi_{L/E}(s)}\]
    is surjective if and only if (a) $L/E$ is unramified and $s \geq 0$ or (b) $\Group(L/E) \neq 1$ and $s > \ell(L/E)$.
    \item \label{it:norm2}  Suppose  $s > \ell(L/E)$. If $x \in L_{\geq s}$, then $\Nm_{L/E}(1+x) = 1 +\Tr_{L/E}(x)$ modulo $E^\times_{> \varphi_{L/E}(s)}$.  Since   we know from Corollary~\ref{cor:cdifferentandHH} that $s > \ell(L/E)$ implies $\varphi_{L/E}(s) = s + c_{L/E}$,  we have the following commutative diagram
\[
\begin{tikzcd}
L_{=s}\arrow[rr, "\Tr_{L/E}"]\arrow[d, "\rotatebox{90}{$\sim$}"]& &E_{=(s+c_{L/E})}\arrow[d, "\rotatebox{90}{$\sim$}"]\\
L^{\times}_{=s}\arrow[rr, "\Nm_{L/E}"]& &E^{\times}_{=(s + c_{L/E})}
\end{tikzcd}
\]
 in which the vertical maps are the $x \mapsto 1 + x$ isomorphisms.
   \end{enumerate}
\end{lemma}

\begin{proof}
    Suppose first that $L/F$ is unramified.  In this case  the statements in the Lemma follow from~\cite[V.2 Propositions 1 and 3]{serre:local} and their proofs.  
    
    Now suppose  that $L/F$  is totally ramified and cyclic of prime order.  Since
    \[ \Nm_{L/E} \colon L^\times_{= \ell(L/E)} \ra E^\times_{=\varphi_{L/E}(\ell(L/E))} \]
    is neither injective nor surjective~\cite[V.3 Proposition 5(iii)]{serre:local},  it follows that if $s \leq \ell(L/E)$, then  $\Nm_{L/E}[L^\times_{\geq s}] \subsetneq E^\times_{\geq \varphi_{L/E}(s)}$.  On the other hand, for $s > \ell(L/E)$, 
    the proof of~\cite[V.3 Proposition 5(iv)]{serre:local}  shows that Claim~\ref{it:norm2} of the Lemma holds.  It follows that  for all $s > \ell(L/E)$ we have $\Nm_{L/E}[L^\times_{\geq s}] = E^\times_{\geq \varphi_{L/E}(s)}$.

   Suppose now that $L/E$ is a totally ramified Galois extension.  We shall proceed by induction on $|\Group(L/E)|$.  If $\Group(L/E)$ is trivial, then $E = L$ and there is nothing to prove.  If $\Group(L/E)$ is not trivial, then  since $\Group(L/E)$ is solvable there is a subextension $K/E$ of $L/E$ such that $\Group(K/E) = \Group(L/E)/\Group(L/K)$ is cyclic of prime order.  By induction we know that for all $r > 0$ 
   \begin{enumerate}
   \item 
    \(  \Nm_{L/K} [L^\times_{\geq r}] \subset  K^\times_{\geq \varphi_{L/K}(r)}\) with equality if and only if $r > \ell(L/K)$ and 
    \item \label{it:norm2a} if  $r > \ell(L/K)$ and  $x \in L_{\geq r}$, then $\Nm_{L/K}(1+x) = 1 +\Tr_{L/K}(x)$ modulo $K^\times_{> \varphi_{L/K}(r)}$. 
    \end{enumerate}
   Since $K/E$ is cyclic of prime order, we know that for all $t > 0$
    \begin{enumerate}
   \item 
    \(  \Nm_{K/E} [K^\times_{\geq t}] \subset  E^\times_{\geq \varphi_{K/E}(t)}\) with equality if and only if $t > \ell(K/E)$ and 
    \item \label{it:norm2b} if  $t > \ell(K/E)$ and  $y \in K_{\geq t}$, then $\Nm_{K/E}(1+y) = 1 +\Tr_{K/E}(y)$ modulo $E^\times_{> \varphi_{K/E}(t)}$. 
    \end{enumerate}
For all $s > 0$ we have 
\begin{equation*}
    \begin{split}
    \Nm_{L/E}[L^\times_{\geq s}] &= \Nm_{K/E}[\Nm_{L/K}[L^\times_{\geq s}]] \\ 
    &\subset \Nm_{K/E}[K^\times_{\geq \varphi_{L/K}(s)}] \\
    &\subset E^\times_{ \geq \varphi_{K/E}(\varphi_{L/K}(s))} \\
    &= E^\times_{\geq\varphi_{L/E}(s)} 
     \end{split}
\end{equation*}
 with equality if and only if both $s> \ell(L/K)$ and $\varphi_{L/K}(s) > \ell(K/E)$.  But, from the exact sequence~(\ref{eq:exactnumberlower}) we know that  $s> \ell(L/K)$ and $\varphi_{L/K}(s) > \ell(K/E)$ if and only if $s > \ell(L/E)$.  Thus, Claim~\ref{it:norm1} holds for totally ramified Galois extensions. Finally, if $r > \ell(L/E)$, then for all $x \in L_{\geq r}$ we have
 \begin{equation*}
     \begin{split}
         \Nm_{L/E}(1 + x) & = \Nm_{K/E} ( \Nm_{L/K}(1 + x)) \\
         &= \Nm_{K/E} (1 + \Tr_{L/K}(x) + z) \qquad \hphantom{+ w}\text{     for some $z \in K_{>\varphi_{L/K}(r)}$}\\
          &= 1 + \Tr_{K/E}(\Tr_{L/K}(x) + z) + w  \qquad \text{ for some $w \in E_{>\varphi_{K/E}(\varphi_{L/K}(r))}$}\\
          &= 1 + \Tr_{L/E}(x)  \hphantom{Go} \qquad  \qquad \qquad \qquad \text{ modulo $E^\times_{>\varphi_{L/E}(r)}$,}\\
     \end{split}
 \end{equation*}
    and so Claim~\ref{it:norm2} holds for totally ramified Galois extensions.

 Finally, suppose $L/E$ is any finite Galois extension.  The case when $L/E$ is unramified is handled above, so we may assume $\Group(L/E)$ is nontrivial.    We let $K = L^{\Group(L/E)}$ be the maximal unramified extension of $E$ inside $L$.  Since $L/K$ is totally ramified with Galois group $\Group(L/K) = \Group(L/E)$ and $\varphi_{L/E} = \varphi_{L/K}$,  from our results above we have that for all $s > 0$
 \[ \Nm_{L/E}[L^\times_{\geq s}] = \Nm_{K/E}[\Nm_{L/K}[L^\times_{\geq s}]] \subset \Nm_{K/E}[K^\times_{\geq \varphi_{L/K}(s)}] = E^\times_{\geq \varphi_{L/K}(s)} = E^\times_{\geq \varphi_{L/E}(s)}\]
 with equality if and only if $s > \ell(L/K) = \ell(L/E)$.   For $r > \ell(L/E)$ our results above similarly yield $\Nm_{L/E}(1+x) = 1 + \Tr_{L/E}(x)$ modulo $E^\times_{>\varphi_{L/E}(r)}$.
\end{proof}

\subsection{Upper numbering subgroups and extensions} Suppose $E$ is an extension of $F$.
Recall that for  $r \geq 0$ we have
\[\Group_E^r = \lim_{\leftarrow} \Group(K/E)^r \text{ and } \Group_E^{r+} = \lim_{\leftarrow} \Group(K/E)^{r+}\]
where the limit is over all finite Galois extensions of $E$.

\begin{lemma} \label{lem:inertiaintersections} Suppose $r \in \R_{\geq 0}$,  and $L/E$ is a finite Galois extension.   We have
\[\Group_E^r \cap \Group_L^0 = \Group_L^{\psi_{L/E}(r)} \text{ and } \Group_E^{r+} \cap \Group_L^0 = \Group_L^{\psi_{L/E}(r)+}.\]
Moreover, if $\Group_E^r \leq \Group^0_L$, then  for all $s \geq r$ we have
\begin{enumerate}
    \item $\Group_L^{\psi_{L/E}(s)} = \Group_E^s$  and
    \item $\psi_{L/E}(s) - s = \psi_{L/E}(r) - r$.
\end{enumerate} 
\end{lemma}

\begin{proof}
Suppose $M/E$ is a finite Galois extension such that $L \leq M$.

For the first statement, it will be enough to show that $\Group(M/L)^{\psi_{L/E}(s)} = \Group(M/E)^s \cap \Group(M/L)$ for all $s \geq r$.   Suppose $s \geq r$. We have
\begin{equation*}
    \begin{split}
       \Group(M/E)^s \cap \Group(M/L) 
        &= \Group(M/E)_{\psi_{M/E}(s)} \cap \Group(M/L) \\
        &= \Group(M/L)_{\psi_{M/E}(s)} = \Group(M/L)^{\varphi_{M/L} (\psi_{M/E}(s))}
    \end{split}
\end{equation*}
Since $\varphi_{M/E} = \varphi_{L/E} \circ \varphi_{M/L}$, we have 
\[\varphi_{M/L} \circ \psi_{M/E} = \varphi_{M/L} \circ \varphi^{-1}_{M/E} = \varphi_{M/L} \circ (\varphi_{L/E} \circ \varphi_{M/L})^{-1} = \varphi^{-1}_{L/E} = \psi_{L/E}.\]

Now assume $\Group_E^r \leq \Group_L^0$.   The first of the remaining statements follows  from our work above.  The calculation above   shows that
for all $t \geq \psi_{M/E}(r)$ we have
\( \Group(M/E)_t = \Group(M/L)_t\).
Thus
\begin{equation*}
    \begin{split}
        \psi_{L/E}(s) - s &= \varphi_{M/L}(\psi_{M/E} (s)) -  \varphi_{M/E}(\psi_{M/E} (s)) \\
        &= \varphi_{M/L}(\psi_{M/E} (r)) + \int_{\psi_{M/E}(r)}^ {\psi_{M/E} (s)}  \big| \Group(M/
        L)_t \big| \, dt \\
        &  \hspace{2em}  - \Bigr( \varphi_{M/E}(\psi_{M/E} (r))  + \int_{\psi_{M/E} (r)}^ {\psi_{M/E} (s)}  \big| \Group(M/E)_t \big| \, dt \Bigl) \\  
        &= \varphi_{M/L}(\psi_{M/E} (r)) - \varphi_{M/E}(\psi_{M/E} (r)) 
        = \psi_{L/E}(r) - r. \qedhere
    \end{split}
\end{equation*}
\end{proof}

\begin{lemma}\label{lem:TFAE}
Suppose $s \in \R_{\geq 0}$.  The following statements are equivalent.
\begin{enumerate}

 \item \label{item:six}   $ c_{L/E} = s - \psi_{L/E}(s) $.

       \item   \label{item:five} \( \psi_{L/E}(s) \geq \ell(L/E)\)  or \( s \geq u(L/E)\).
       
    \item  \label{item:two}  \(\Group(L/E)_{\psi_{L/E}(s)+} = 1 \).
    
\item  \label{item:three} \( \Nm_{L/E} [L^\times_{> \psi_{L/E}(s)}] = E^\times_{>s} \).

\item    \label{item:four} $\Group_L^{\psi_{L/E}(s)+} = \Group_E^{s+}$. 

   \item \label{item:one} \(\Group_E^{s+} \leq \Group^0_L    \).

\end{enumerate} 
\end{lemma}

\begin{proof} 

(\ref{item:six}) $\Leftrightarrow$ (\ref{item:five}):  From Corollary~\ref{cor:cdifferentandHH} we know that $c_{L/E} = u(L/E) - \ell(L/E)$.  Since $t \mapsto (t - \psi_{L/E}(t))$ is a nondecreasing function on $\R_{\geq 0}$,  for all $s \geq 0$ we have $s - \psi_{L/E}(s) = u(L/E) - \ell(L/E)$ if and only if $s \geq u(L/E)$. 
Since $s \geq u(L/E)$ if and only if $\psi_{L/E}(s) \geq \ell(L/E)$, the result follows.

(\ref{item:five}) $\Leftrightarrow$ (\ref{item:two}):
Since  \( \psi_{L/E}(s) \geq \ell(L/E)\)  if and only if \( s \geq u(L/E)\), it is enough to show (\ref{item:two}) holds if and only if  \( \psi_{L/E}(s) \geq \ell(L/E)\).  This follows  from the definition of $\ell(L/E)$.

(\ref{item:two}) $\Leftrightarrow$ (\ref{item:three}): 
Since $\Group(L/E)_{\psi_{L/E}(s)+} =1$ if and only if $\psi_{L/E}(s) \geq \ell(L/E)$, this follows from  Lemma~\ref{lem:normrange}~(\ref{it:norm1}).

(\ref{item:two}) $\Leftrightarrow$ (\ref{item:four}):
Suppose $M/E$ is a finite Galois extension with $L \leq M$.   We need to show \(\Group(M/L)^{\psi_{L/E}(t)} = \Group(M/E)^t \) if and only if $\Group(L/E)^t = \Group(L/E)_{\psi_{L/E}(t)}$ is trivial.
Thanks to Equation~(\ref{eq:exactnumberupper})
\[ 1 \ra \Group(M/L)^{\psi_{L/E}(t)} \ra \Group(M/E)^t  \ra \Group(L/E)_{\psi_{L/E}(t)} \ra 1\]
is exact.   From this it follows that $\Group(L/E)_{\psi_{L/E}(t)}$ is trivial if and only if 
\[\Group(M/L)^{\psi_{L/E}(t)} = \Group(M/E)^t.\]

(\ref{item:four}) $\Leftrightarrow$ (\ref{item:one}):  
From Lemma~\ref{lem:inertiaintersections} we know that $\Group_L^{\psi_{L/E}(s)+}  = \Group_E^{s+} \cap \Group_L^0$.  Since $\Group_L^0 \leq \Group_E^0$, this means $\Group_L^{\psi_{L/E}(s)+} \leq \Group_E^{s+}$.    Thus, 
$\Group_L^{\psi_{L/E}(s)+}  = \Group_E^{s+}$ if and only if $\Group_E^{s+} \leq \Group_L^{\psi_{L/E}(s)+}$. 

If   
 $\Group_E^{s+} \leq \Group_L^{\psi_{L/E}(s)+}$, then  $\Group_E^{s+} \leq \Group_L^{0}$.  On the other hand, if $\Group_E^{s+} \leq \Group_L^{0}$, then $\Group_E^{s+} = \Group_E^{s+} \cap \Group_L^{0} = \Group_L^{\psi_{L/E}(s)+}$ from Lemma~\ref{lem:inertiaintersections}, so $\Group_E^{s+} \leq \Group_L^{\psi_{L/E}(s)+}$.
\end{proof}

\subsection{An existence result}
We will often need to know that there exists an extension $L/E$ that has reasonable properties with respect to the depth of the representation we are considering.  In particular, if the depth is $s$, then we will need that $s \cdot e(L/E) \in \Z$ and $u(L/E) < s$.  The next Lemma shows that such an extension exists.

\begin{lemma}  \label{lem:Eexists} Suppose $n\in\Z_{>0}$ and  $\varepsilon\in\R_{>0}$. There exists a finite Galois extension $L/E$ such that
\begin{enumerate}
    \item $n$ divides $e(L/E)$,
    \item $u(L/E) < \varepsilon$, and 
    \item $\Group_L^{\varepsilon}\subset \Group_E^{0+}$.
\end{enumerate}
Moreover, if $L/E$ satisfies these three conditions, then so too does $LM/E$ where $M$ is any finite tame Galois extension of $E$.
\end{lemma}

\begin{proof}
If $n=1$, then we can take $L=E$.  Assume now that $n > 1$.
     Since $\Group_E^{\varepsilon}$ is a closed normal subgroup of $\Group_E = \Group_E^0$, the quotient $\Group_E/\Group_E^{\varepsilon}$ is a profinite group~\cite[V.1 Corollary~3]{cassels-frohlich}.   Write $n = n'p^m$ with $(n',p) =1$.  Thanks to Euler's theorem, $n'$ divides $p^{\phi(n')j} - 1$ for all $j\in \Z_{>0}$.  Thus, we can find a tamely ramified extension $\tilde{E}$ of $E$ such that $n = n'p^m$ divides $| \tilde{E}^{\times}_0/\tilde{E}^{\times}_\varepsilon| = | \tilde{E}^{\times}_0/\tilde{E}^{\times}_{0^+}| \cdot  | \tilde{E}^{\times}_{0^+}/\tilde{E}^{\times}_\varepsilon |$.  
       Since $\Group_{\tilde{E}} \leq \Group_E$ and $\Group_{\tilde{E}} \cap \Group_E^\varepsilon = \Group_{\tilde{E}}^{\psi_{\tilde{E}/E}(\varepsilon)} = \Group_{\tilde{E}}^{\varepsilon}$,
      we have an injective map $\Group_{\tilde{E}}/\Group_{\tilde{E}}^{\varepsilon} \ira \Group_E/\Group_E^{\varepsilon}$.  On the other hand, from local class field theory we have a surjection $\Group_{\tilde{E}}/\Group_{\tilde{E}}^{\varepsilon}  \sra \tilde{E}^{\times}_0/\tilde{E}^{\times}_\varepsilon$.  Thus, $n$ divides the (pro-)order of $\Group_E/\Group_E^{\varepsilon}$ as a profinite group.  Consequently, there exists a finite Galois extension $M$ of $E$ such that $n$ divides $| \Group({M}/E)/\Group({M}/E)^\varepsilon|$. 
      Let $L = {M}^{\Group({M}/E)^\varepsilon}$.

 Since $\Group(L/E) \cong \Group({M}/E)/\Group({M}/E)^\varepsilon$, we know that $n$ divides $e(L/E) = |\Group(L/E)|$. 

      Choose $u \in \R_{\geq 0}$ such that $\Group(M/E)^{u:u+}$ is not trivial and $\Group(M/E)^{u+} = \Group(M/E)^\varepsilon$.  (Since the assignment $s \mapsto \Group(M/E)^s$ is right-continuous and  $\Group(L/E)$ is not trivial, such a $u$ exists.)  Note that $\Group(L/E)^s \cong \Group(M/E)^s \Group(M/E)^\varepsilon/\Group(M/E)^\varepsilon$, which is trivial if $s > u$  since then $\Group(M/E)^s \leq \Group(M/E)^{u+} = \Group(M/E)^\varepsilon$.   Thus, $\varepsilon > u  \geq u(L/E)$.

      Since $ \varepsilon > u \geq u(L/E)$, from Lemma~\ref{lem:TFAE} we conclude that $\Group_L^{\varepsilon} \leq \Group_L^{u+} = \Group_E^{\psi_{L/E}(u)+} \leq \Group_E^{0+}$.

      Finally, suppose $M$ is a finite tame extension of $E$.  Since $e(L/E)$ divides $e(LM/E)$, we have that $n$ divides $e(LM/E)$.  Since $LM/L$ is tamely ramified, we have
      \[\psi_{LM/E} = \psi_{LM/L} \circ \psi_{L/E} = \psi_{L/E}.\]
      Thus, the upper numbering breaks of $LM/E$ are exactly those of $L/E$, and we conclude that $u(LM/E) = u(L/E) <\varepsilon$.   Since $\Group_{LM}^{\varepsilon} =\Group_{L}^{\varepsilon}$, we conclude
     that $\Group_{LM}^{\varepsilon}\subset \Group_E^{0+}$.
\end{proof}

\section{A construction for tamely ramified tori}
\label{sec:constructiontametori}
 Suppose $T$ is an $F^t$-split torus, $r \in \Q_{> 0}$,  and  $X \in \ft^*(\bar{F})_{=-r}$.
In this section  we associate to this data   a continuous homomorphism $\phi^T_X \colon \Group_F^r \ra T^{\vee}$ that
is trivial on $\Group_F^{r+}$.

\subsection{$(X,r,T)$-adapted extensions}
We begin with a definition.

\begin{definition}  A finite Galois extension $E/F$ for which 
\begin{itemize}
    \item  $r \cdot e(E/F) \in \Z$,
    \item $u(E/F) < r$,
     \item $T$ is $E$-split, and
    \item $X \in \ft^*(E)_{=-r }$
\end{itemize}
will be said to be \emph{$(X,r,T)$-adapted}.
\end{definition}

\begin{lemma}  \label{lem:lemma12}
    There is a finite Galois extension $E/F$ that is $(X,r,T)$-adapted.
\end{lemma}

\begin{proof}
From Lemma~\ref{lem:Eexists} there is a finite Galois extension $E/F$ such that 
$r \cdot e(E/F) \in \Z$ and $u(E/F) < r$.   Since neither of these properties changes if we replace $E$ with a tamely ramified extension, we may also assume that $T$ is an $E$-split $E$-torus.
   Since $\ft^*(\bar{F})_{=0}$  may be identified with $\ft^*(E^u)_{=0}$ and  $\ft^*(E')_{=-r} \cong \ft^*(E')_{=0}$ for all unramified extensions $E'/E$, after replacing $E$ with a finite unramified extension we may assume that $X \in \ft^*(E)_{=-r}$.
\end{proof}

\begin{remark} If $r \in \Z_{(p)}$, then we can take $E$ to be tamely ramified  (in which case $u(E/F) = 0$).  
\end{remark}

\subsection{The character $\phi^T_{X,E}: \Group_F^{r:r+} \ra T^{\vee}$} \label{sec:constructphi}
Suppose $E/F$ is $(X,r,T)$-adapted.  Consider
\begin{equation}\label{eq:characteronlieT}
\begin{array}{ccc}
\ft(E)_{r-c_{E/F}}&\longrightarrow &\Cc\\
V&\longmapsto&\Lambda_E(V,X).
\end{array}
\end{equation}
where $\Lambda_E$ is the additive character of $E$ defined in Section~\ref{sec:di9fferentdifferent} and $c_{E/F}$ is defined in Definition~\ref{defn:cdefined}.  Denote by $\chi^T_{X,E}$ the character on $T(E)_{\geq(r-c_{E/F})}$ given by the pullback of \eqref{eq:characteronlieT} under the Moy-Prasad isomorphism $T(E)_{=r-c_{E/F}}\xra{\sim}\ft(E)_{=r-c_{E/F}}$.

Suppose $\psi_i$ for $i \in \{1,2\}$ is a character of $T(E)$ whose restriction to $T(E)_{\geq r-c_{E/F}}$ yields $\chi_{X,E}^T$. 
Denote the associated Langlands parameter by $\varphi_i \colon W_E\ra {}T^{\vee}$.  Since  $\psi_1^{-1} \psi_2$ has depth less than $r-c_{E/F}$, by the local Langlands correspondence for tori and its preservation of depth for tamely ramified tori~\cite[\S7.5 and \S7.10]{Yu09} we conclude that  $\varphi_1^{-1} \varphi_2$ will also have depth less than $r-c_{E/F}$.  That is, $\varphi_1$ and $\varphi_2$ have the same restriction to $\Group_E^{r-c_{E/F}}$.   Thus,  there is a bijective correspondence between the group of characters of $T(E)_{=r-c_{E/F}}$ and the group of continuous homomorphisms $\Group_E^{r-c_{E/F}}/ \Group_E^{(r-c_{E/F})+}\ra T^{\vee}$ that can be extended to an $L$-homomorphism $W_E\ra {}T^{\vee}$.

Since $r > u(E/F)$, from Lemma~\ref{lem:TFAE} we have $\psi_{E/F}(r) = r - c_{E/F}$, 
$ \Group_E^{r - c_{E/F}}  = \Group_F^{r}$, and $\Group_E^{(r-c_{E/F})+} = \Group_F^{r+}$.
Thus,  the bijection in the paragraph above associates to  $\chi^T_{X,E}$ a continuous homomorphism that we will denote by $\phi^T_{X,E}: \Group_F^{r:r+} \ra T^{\vee}$.  We will also denote by $\phi^T_{X,E}$ the inflation of this homomorphism to a homomorphism  $\Group_F^r \ra T^{\vee}$. 
We have an embedding $T^{\vee}\ira G^{\vee}$ that is determined up to $G^{\vee}$-conjugation,
and the $G^{\vee}$-conjugacy class of $\phi_{X,E}^T$ is a tame restricted depth-$r$ parameter as discussed in the Introduction.

\begin{remark} \label{rem:wildtori} For tori that are not tamely ramified, it is known that depth is not preserved~ (see~\cite{AP18}, ~\cite{MP19} or the comments following~\cite[Theorem~3.33]{AP22}).  Hence, the construction above cannot be carried out without a tameness assumption.   In the known examples where depth preservation fails for tori, the depth of the character is smaller than the depth of the corresponding Langlands parameter.    
\end{remark}

\subsection{$\phi^T_{X,E}$ is independent of the choice of $E$}  
Suppose that
 $E$ and $E'$ are both finite separable extensions of $F$ that are $(X,r,T)$-adapted.
In this section we show that $\phi_{X,E}^T = \phi_{X,{E'}}^T$.

\begin{lemma} \label{lem:adatptedreduction} 
The compositum of $(X,r,T)$-adapted fields is again
 $(X,r,T)$-adapted.
\end{lemma}

\begin{proof}
It will be enough to show that $u(EE'/F) < r$.
Since $E$ and $E'$ are Galois, the assignment $\sigma  \mapsto (\res_E \sigma, \res_{E'}\sigma)$  yields an injective map $\Group(EE'/F) \ira \Group(E/F) \times \Group(E'/F)$.  Thus, for all $s \in \R_{\geq 0}$ we have an injective map $\Group(EE'/F)^s \ira \Group(E/F)^s \times \Group(E'/F)^s$.  It follows that $u(EE'/F) \leq \max(u(E/F),u(E'/F)) < r$.
\end{proof}

Thanks to Lemma~\ref{lem:adatptedreduction},  it will be enough to show $\phi^T_{X,E} = \phi^T_{X,E'}$  when  $E'$ is a finite Galois extension of $E$.   So, we may and do assume that   $E'$ is a finite Galois extension of $E$.

Since for all $s \in \R_{\geq 0}$ we have $\Group(E'/F)_s \cap \Group(E'/E) = \Group(E'/E)_{s}$, it follows that  $\ell(E'/E) \leq \ell(E'/F)$.
 From Corollary~\ref{cor:cdifferentandHH} and the fact that $u(E'/F) < r$ 
we have
\[ \ell(E'/E) \leq \ell(E'/F) = u(E'/F) - c_{E'/F} < r - c_{E'/F}.\]
Thus,
from Lemma~\ref{lem:normrange} 
and the fact that $c_{E'/F}  = c_{E'/E} + c_{E/F}$, the trace map and the norm map behave well with respect to the subgroups of interest:
\[
\Tr_{E'/E}:\ft(E')_{\ge r - c_{E'/F}} \sra\ft(E)_{\ge r-c_{E/F}},\;\;\Nm_{E'/E}:T(E')_{\ge r - c_{E'/F}}\sra T(E)_{\ge r - c_{E/F}}.
\] 
These maps induce maps on the quotients which we denote by the same name
\[
\Tr_{E'/E}:\ft(E')_{=r-c_{E'/F}}\ra\ft(E)_{=r-c_{E/F}}\;\;\Nm_{E'/E}:T(E')_{=r-c_{E'/F}}\ra T(E)_{=r-c_{E/F}}.
\]
We then have the following commutative diagram in which the vertical maps are the Moy-Prasad isomorphisms
\[
\begin{tikzcd}
\ft(E')_{=r-c_{E'/F}}\arrow[r, "\Tr_{E'/E}"]\arrow[d, "\rotatebox{90}{$\sim$}"]&\ft(E)_{=r-c_{E/F}}\arrow[d, "\rotatebox{90}{$\sim$}"]\\
T(E')_{=r-c_{E'/F}}\arrow[r, "\Nm_{E'/E}"]&T(E)_{=r-c_{E/F}}
\end{tikzcd}
\]
We have proved the following lemma

\begin{lemma}\label{lem:norm}  
Suppose $F \leq E \leq E'$ is a tower of finite Galois extensions such that $E$ and $E'$ are $(X,r,T)$-adapted.  We have $\chi_{X,E'}^T=\chi_{X,E}^T\circ\Nm_{E'/E}$.
\end{lemma}

The following Lemma allows us to pass to the 
Galois side.

\begin{lemma}\label{lem:LCFT}  In this lemma there are no tameness restrictions on $T$.   Let $E'/E$ be a finite Galois extension of non-archimedean local fields and let $T$ be a torus over $E$. Consider $\Nm_{E'/E}:T(E')\ra T(E)$. Suppose $\chi:T(E)\ra\Cc$ is a character and $\varphi\in H^1(W_E,T^{\vee})$ is its Langlands parameter. Then $\res_{W_{E'}} \varphi$ is the Langlands parameter of $\chi\circ\Nm_{E'/E}:T(E')\ra\Cc$.
\end{lemma}

\begin{remark}
    While preparing this manuscript, a  version of Lemma~\ref{lem:LCFT} appeared in~\cite[Prop. 2.1]{SX25}.
\end{remark}

\begin{proof} 
Following \cite[\S7.6]{Yu09} put $S=\Res_{L/E}(T\times_EL)$ for some finite Galois extension $L/E$ splitting $T$.  We may and do assume that $L$ contains $E'$.   We have $T\ira S=(\Res_{L/E}\mathbb{G}_m)^{\dim T}$.  Since $T(E)$ injects into $S(E)$, 
our character $\chi$ extends to some $\til{\chi}:S(E)\ra\Cc$, which corresponds by local Langlands for $S$ to some $\til{\varphi}\in H^1(W_E,S^{\vee})$. Functoriality of local Langlands for tori then says that $\varphi$
is the composition of $\til{\varphi}$ with $S^{\vee}\sra T^{\vee}$.
The desired assertion in the lemma for $\chi$ and $\varphi$ then follows from that of $\til{\chi}$ and $\til{\varphi}$. That is, we may assume $T$ is of the form $T=\Res_{L/E}\mathbb{G}_m$.

Choose a set $\sigma=(\sigma_{\bar{g}})$ of representatives of the left $W_{E'}$-cosets in $W_E$, so $W_E = \bigsqcup \sigma_{\bar{g}} W_{E'}$. 
Local Langlands for $T$ as an $E$-torus matches characters on $T(E)=L^{\times}$  with elements in $H^1(W_E,T^{\vee})=H^1(W_E,\Ind_{W_L}^{W_E}\Cc)\cong H^1(W_L,\Cc)$ via Shapiro's lemma and local class field theory.
Similarly, Local Langlands for $T$ as an $E'$-torus matches characters on $T(E')=\prod_{\sigma} L^{\times}$ with elements in 
\begin{equation*}
    \begin{split}
         H^1(W_{E'},T^{\vee}) &=H^1(W_{E'},\Ind_{W_L}^{W_E}\Cc)
=H^1(W_{E'},\Ind_{W_L}^{W_{E'} }\Ind_{W_{E'}}^{W_E}\Cc)  \\
&\cong H^1(W_{L},\Ind_{W_{E'}}^{W_E}\Cc) 
=  \bigoplus_{\sigma} H^1(W_{L},\Cc).
 \end{split}
\end{equation*}

Moreover, tracing through these isomorphisms shows that the following diagram
\[
\begin{tikzcd}
 H^1(W_E,T^{\vee})\arrow[r, "\sim"]\arrow[d] & H^1(W_L,\Cc)\arrow[d] \\
H^1(W_{E'},T^{\vee})\arrow[r, "\sim"] & \bigoplus_{\sigma} H^1(W_L,\Cc) 
\end{tikzcd}
\]
commutes  where the left vertical map is the natural restriction map and the right vertical map sends $\alpha$ to $(\sigma_{\bar{g}}.\alpha)$.
We conclude that 
the map  from $\widehat{T(E)}$ to $\widehat{T(E')}$ that is dual to the restriction map $ H^1(W_E,T^{\vee}) \ra H^1(W_{E'},T^{\vee})$ is given by $\chi \mapsto \chi \circ \Nm_{E'/E}$ where  $\Nm_{E'/E} \colon \prod_{\sigma}(L^{\times})\cong T(E')\ra T(E)=L^{\times}$ is given by  $(\ell_{\sigma_{\bar{g}}})\mapsto\prod_{}\sigma_{\bar{g}}.\ell_{\sigma_{\bar{g}}}$.
\end{proof}

\begin{corollary}\label{cor:base}  If
 $E/F$ and $E'/F$ are two $(X,r,T)$-adapted Galois extensions,
Then $\phi_{X,E}^T = \phi_{X,{E'}}^T$.
\end{corollary}

\begin{proof} It is enough to show this when $E'/E$ is Galois. The assertion follows from Lemma~\ref{lem:norm} and Lemma~\ref{lem:LCFT}.
\end{proof}

\subsection{$T$-toral tame Moy-Prasad types of depth $r$}  \label{sec:Ttoraltype}

 Suppose  $(x,X) \in \MP_r$ and   $T$ is a maximal $F^t$-split torus in $G$.   We will say that $(x,X)$ is \emph{$T$-toral}    provided that there exists  a finite Galois extension $E$ of $F$  such that 
\begin{itemize} 
\item $T$ is defined over $E$,
\item $E$ is $(X,r,T)$-adapted, and
\item $x$ belongs to the building of $T(E^t)$ in $\cB(G(E^t))$.
\end{itemize}
Recall that $X \in \fg^*(F)_{x=-r}$, so   $X\in\ft^*({E})_{=-r}$ means $X\in\ft^*(E)_{=-r}\cap\fg^*(F)_{x=-r}$ where the intersection is taken in $\fg^*(E)_{x=-r}$.

The continuous homomorphism $\phi^T_X =\phi_{X,E}^T \colon \Group_F^r \ra T^{\vee}$  constructed  in Section~\ref{sec:constructphi} above is what we propose to associate to the pair $(x,X)$.  More precisely:

\begin{conjecture}\label{conj:desi-pre} Recall that $G$ splits over $F^t$.  Suppose an irreducible smooth representation $\pi$ of $G(F)$ has a nondegenerate  $T$-toral  Moy-Prasad type $(x,X)$ of depth $r>0$ with $r\in \Q_{>0}$.  Choose a finite $(X,r,T)$-adapted
extension $E$ of $F$ such that $T$ is an $E$-torus, $x$ belongs to the building of $T(E^t)$ over $E$, and  $X\in\ft^*(E)_{=-r}$. Then the restriction of a
Langlands parameter for $\pi$ to $\Group_F^r$ is $G^{\vee}$-conjugate to $\phi_{X}^T  = \phi_{X,E}^T$.
\end{conjecture}

\begin{remark} \label{rem:indofaddchar} The additive character $\Lambda_F:k_F\ra\Cc$ appears twice in Conjecture \ref{conj:desi-pre}, namely in \eqref{eq:MPtypeongroup} and in \eqref{eq:characteronlieT}.  The two appearances cancel out, and it follows that the validity of Conjecture \ref{conj:desi-pre} is independent of the choice of $\Lambda_F$.
\end{remark}

\section{Depth-\texorpdfstring{$r$}{r} Deligne-Lusztig parameters}
\label{sec:MPtypetoDLparam}
Suppose $(x,X)$ is a Moy-Prasad type of depth $r$.  Since this section is concerned only about the Lie algebra and its dual, we can assume $r \in \Q$ (rather than $r \in \Q_{>0}$).

We present a way to pass from the Moy-Prasad type $(x,X)$ to data on the Galois side.
This proceeds in two steps: in this section we construct some data on the $p$-adic side, called depth-$r$ Deligne-Lusztig parameters, and in Section~\ref{sec:C2Bnew} we match this data bijectively with some data on the Galois side.

Recall that $c_{\Lambda_F} = 0$.

\subsection{A construction of depth-$r$ Deligne-Lusztig parameters}  \label{constructionofDL}
Suppose $x \in \cB(F)$ and $r \in \Q$.
  Our goal in Section~\ref{constructionofDL} is to construct a map from $\fg^*(F)_{x = -r}$ to $\bar{\fa}^*  \quo W$ where $\bar{\fa}^* $ is the dual space of the Lie algebra of the reductive quotient of $A(F^t)$ and $W$ is the Weyl group.

Suppose $(x,X) \in \MP_r$.

\begin{definition}  A pair $(\alpha, E)$, where $E/F$ is a finite Galois extension and $\alpha \in E$, is said to be \emph{$(x,X)$-adapted}   provided that 
\begin{enumerate}
    \item   $r \cdot e(E/F) \in \Z$, 
    \item  $u(E/F) < r$,
    \item $G$ splits over $E$, and 
    \item $\val(\alpha) = r$.
\end{enumerate}
\end{definition}

 Thanks to Lemma~\ref{lem:Eexists}, we can choose a finite Galois extension
$E/F$  such that $r \cdot e(E/F) \in \Z$ and
$u(E/F) < r$.
Since $G$ splits over a tamely ramified extension, thanks to Lemma~\ref{lem:Eexists} we may replace $E$ with $LE$ where $L$ is a finite tame extension of $F$ over which $G$ splits.  Since $G$ splits over $E = LE$, so too does  $A$~\cite[Lemma~1.4.2]{debacker:totally}.   Choose $\alpha \in E$ such that $\val(\alpha) = r$.  Note that the pair $(\alpha, E)$ is $(x,X)$-adapted.

Choose $h \in G(F)$ such that $hx \in \cA$.

 Let $\bar{A}$ denote the reductive quotient of $A(F^t)$, this is  a torus over $\bar{k}$.  Since $A$ is $E$-split, we (canonically) identify $\bar{A}$ with the reductive quotient of $A(E^u)$.   Write $\fa^*:=(\Lie A)^*$ and $\bar{\fa}^* :=(\Lie\bar{A})^*$ for the  dual spaces of $\Lie A$ and $\Lie\bar{A}$, respectively. The torus $\bar{A}$ is naturally a maximal torus of the reductive quotient $G(E^u)_{hx=0}$. Consider $W:=N_G(A)/A$ and $W^E_{hx}:=N_{G(E^u)_{hx=0}}(\bar{A})/\bar{A}$  as abstract groups.   
Note that $W^E_{hx}$ agrees with the image of $N_{G(E^u)_{{hx}\ge 0}}(A)$ in $W$ (see~\cite[Lemma~7.2.1]{debacker:totally}). We have an embedding $\bar{\fa}^*\ira\fg^*(E^u)_{hx=0}$ given by extending trivially on the nontrivial root spaces in $\fg(E^u)_{hx = 0}$ with respect to $\bar{A}$. This induces a morphism of varieties over $\bar{k} = \bar{k}_E$
\begin{equation}\label{eq:dualChnew}
\bar{\fa}^*\quo W^E_{hx}\ra\fg^*(E^u)_{hx=0}\quo G(E^u)_{hx=0}.
\end{equation}
By \cite[Thm. 4]{KW76} (for $p>2$) and \cite[Thm. 1]{ST25} (for $p=2$), the map \eqref{eq:dualChnew} is a bijection on $\bar{k}$-points, in fact an isomorphism when $p>2$. Let $E'$ denote the maximal tamely ramified subextension in $E/F$ and consider the following commutative diagram:
\begin{equation}\label{eq:squaresnew}
\begin{tikzcd}
 \fg^*(F)_{x=-r}\arrow[d,hook,"\Int(h)"]& &&\\
 \fg^*(E')_{hx=-r}\arrow[d,hook]& & &\\
   \fg^*(E^u)_{hx=-r}\arrow[r,"\times\alpha"]&\fg^*(E^u)_{hx=0}\arrow[r]&\fg^*(E^u)_{hx=0}\quo G(E^u)_{hx=0}\\
     \fa^*(E^u)_{=-r}\arrow[r,"\times\alpha"]\arrow[u, hook]&\fa^*(E^u)_{=0}=\bar{\fa}^*\arrow[u, hook]\arrow[r]&\bar{\fa}^*\quo W^E_{hx}\arrow[u, "\sim"'{sloped, anchor=north}]\arrow[u, "\eqref{eq:dualChnew}"]\arrow[d, two heads]\\
 &  &\bar{\fa}^*\quo W\\
\end{tikzcd}
\end{equation}
The map $\Int(h)$ takes $\bar{Y} \in   \fg^*(F)_{x=-r}$ to $\overline{\Ad^*(h)Y} \in  \fg^*(F)_{hx=-r} \subset  \fg^*(E')_{hx=-r}$.  This last inclusion holds because $E'$ is a tame extension of $F$.  The inclusion $\fg^*(E')_{hx=-r} \subset     \fa^*(E^u)_{=-r}$ holds because $G$ is $E'$-split.
Denote by $i_{E,h,\alpha,x}:\fg^*(F)_{x=-r}\ra\bar{\fa}^*\quo W$ the composition (through the inverse of \eqref{eq:dualChnew}).

\begin{lemma}  \label{lem:indofh}
    The map  $i_{E,h,\alpha,x}:\fg^*(F)_{x=-r}\ra\bar{\fa}^*\quo W$ is independent of the choice of $h \in G(F)$.  Thus, we will drop the $h$ from the notation and denote the map by $i_{E,\alpha,x}:\fg^*(F)_{x=-r}\ra\bar{\fa}^*\quo W$.
\end{lemma}

    \begin{proof}
        Suppose $h' \in G(F)$ such that $h'x$ also  belongs to $\cA$.

Since $hx$ and $h'x$ both belong to $\cA$ and are conjugate by an element of $G(F)$, there exists $m \in N_{G(F)}(A)$ such that $mh'x = hx$.    Since $(mh') h^{-1} \in \Stab_{G(F)}(hx)$, from the affine Bruhat-decomposition there exist $n \in N_{\Stab_{G(F)}(hx)}(A)$ and $\ell \in G(F)_{hx \geq 0}$ such that $(mh') h^{-1} = n \ell$.   That is, $n^{-1} m h' = \ell h$.

Since $n^{-1} m \in N_{G(F)}(A)$ and the target space of both $i_{E,n^{-1}m h',\alpha,x}$ and $ i_{E,h',\alpha,x}$  is $\bar{\fa}^*\quo W$, we observe that  $i_{E, n^{-1}m h',\alpha,x} = i_{E,h',\alpha,x}$.   Thus, we may and do assume $h' = \ell h$.  Note that this implies that $h'x = hx$.

Choose $Y \in \fg^*(E^u)_{x \geq -r}$.
Since the images of $\overline{\alpha \Ad^*(h)Y}$ and $\overline{\alpha \Ad^*(h')Y} = \Ad^*(\bar{\ell})(\overline{\alpha \Ad^*(h)Y)}$ 
 in  $\fg^*(E^u)_{hx=0}\quo G(E^u)_{hx=0}$ agree, we conclude that $i_{E,h',\alpha,x} = i_{E,h,\alpha,x}$.
\end{proof}

\begin{lemma}  \label{lem:indofhandE}
    If $E'/F$ is a finite Galois extension such that $(\alpha, E')$  is $(x,X)$-adapted, then $i_{E,\alpha,x} = i_{E',\alpha,x}$.
    Thus, we will drop the $E$ from the notation and denote the map by $i_{\alpha,x}:\fg^*(F)_{x=-r}\ra\bar{\fa}^*\quo W$.
\end{lemma}

\begin{remark}  The element $\alpha$ is not an arbitrary element in $\bar{F}$ such that $\val(\alpha) = r$. It is required to live in a finite Galois extension $K$ such that $(\alpha, K)$ is $(x,X)$-adapted.
\end{remark}

\begin{proof}
Suppose first that $E'$ is a tamely ramified extension of $E$.  Since $E'/E$ is tamely ramified, it follows that $(\alpha,E')$ is  $(x,X)$-adapted.    Suppose $\bar{Y} \in \fg^*(F)_{x=-r}$.
It will be enough to show $i_{E,\alpha,x}(\bar{Y}) = i_{E',\alpha,x}(\bar{Y})$.
As $\bar{k}$-varieties we have identifications   
$ \fg^*(E^{u})_{hx=-r} \subset \fg^*(E^{\prime u})_{hx=-r}$,  
$\fg^*(E^{u})_{hx=0} \subset \fg^*(E^{\prime u})_{hx=0}$ , and
$G(E^{u})_{hx=0} \leq G(E^{\prime u})_{hx=0} $.
Thus, the $G(E^{ u})_{hx=0}$-orbit of the image of $\bar{Y}$ in $\fg^*(E^{ u})_{hx=0}$ is contained in the $G(E^{\prime u})_{hx=0}$-orbit of the image of $\bar{Y}$ in $\fg^*(E^{\prime u})_{hx=0}$. Since these identifications are compatible with the surjections  $\bar{\fa}^*\quo W^E_{hx} \sra \bar{\fa}^*\quo W^{E^\prime}_{hx} \sra \bar{\fa}^*\quo W$, it follows that  $i_{E,\alpha,x}(\bar{Y}) = i_{E',\alpha,x}(\bar{Y})$.

 Suppose $E'/F$ is a finite separable extension such that $(\alpha,E')$ is $(x,X)$-adapted.  Without loss of generality, we replace $E'$ with the compositum of $E'$ and the maximal tame  subextension of $E$. 
Let $K = E \cap E'$.  Note that $A$ splits over $K$.  We have $\alpha \in E \cap E'$ and since $\val(\alpha) = r$, we conclude that $r \cdot e(K/F) \in \Z$.  Since $K$ is Galois over $F$ and we have surjections $\Group(E/F)^s \sra \Group(K/F)^s$ and $\Group(E'/F)^s \sra \Group(K/F)^s$ for all $s \in \R_{\geq 0}$, we conclude that $u(K/F) \leq \min(u(E/F), u(E'/F)) < r$.   Hence, $(\alpha,K)$ is  $(x,X)$-adapted.  

Suppose $\bar{Y} \in \fg^*(F)_{x=-r}$.
Without loss of generality it will be enough to show $i_{E,\alpha,x}(\bar{Y}) = i_{K,\alpha,x}(\bar{Y})$.
As $\bar{k}$-varieties we have identifications   
$\fg^*(K^{ u})_{hx=-r} \subset \fg^*(E^{u})_{hx=-r}$  and
$\fg^*(K^{ u})_{hx=0} \subset \fg^*(E^{u})_{hx=0}$.
Thus, the $G(K^u)_{hx=0}$-orbit of the image of $\bar{Y}$ in $\fg^*(K^{ u})_{hx=0}$ is contained in the $G(E^u)_{hx=0}$-orbit of the image of $\bar{Y}$ in $\fg^*(E^{ u})_{hx=0}$.  Since these identifications are compatible with the surjections  $\bar{\fa}^*\quo W^K_{hx} \sra \bar{\fa}^*\quo W^E_{hx} \sra \bar{\fa}^*\quo W$, we conclude that $i_{E,\alpha,x}(\bar{Y}) = i_{K,\alpha,x}(\bar{Y})$.
\end{proof}

\begin{remark}
    Suppose $(\alpha, E)$ is $(x,X)$-adapted. 
    If $L$ is any tame
    extension of $F$, then we can construct as above a map $i^L_{\alpha,x} \colon \fg^*(L)_{x=-r} \ra \bar{\fa}^* \quo W$ with $L$ in place of $F$  (so we replace $E$ with $EL$).  We have $i_{\alpha,x} = i^F_{\alpha,x} = i^L_{\alpha,x} \circ \iota_L$ where $\iota_L \colon \fg^*(F)_{x = -r} \ira \fg^*(L)_{x=-r}$ is the  injection discussed in Section~\ref{sec:duallattice}.   Since $\fg^*(F)_{x=-r} \ira \fg^*(E')_{x=-r} \ira \fg^*(E)_{x=-r}$ where $E'$ is the maximal tame subextension in $E/F$, we also have  $i_{\alpha,x} =  i^E_{\alpha,x} \circ \iota_E$ where $\iota_E \colon \fg^*(F)_{x = -r} \ira \fg^*(E)_{x=-r}$ and $i^E_{\alpha,x}  \colon \fg^*(E)_{x=-r} \ra \bar{\fa}^* \quo W$ is the map constructed above with $F$ replaced by $E$.
\end{remark}

\begin{definition}\label{def:DL}  
For any $c\in \bar{k}$
we let $j_c:\bar{\fa}^*\quo W\xra{\sim}\bar{\fa}^*\quo W$ denote the isomorphism induced by $\times c:\bar{\fa}^*\xra{\sim}\bar{\fa}^*$.   
    \begin{enumerate}[label=(\roman*)]
        \item The set of {\bf depth-$r$ Deligne-Lusztig parameters} is
  \[\DLr:=\{(\beta,Z)\;|\;  \beta \in 
  \bar{F},
  \;\val(\beta)=r,\;Z\in(\bar{\fa}^*\quo W)(\bar{k})\}/\!\equiv\]
  where $(\beta,Z)\equiv(\beta',Z')$ provided that $Z'=j_c(Z)$ for  $c = \beta'/\beta +
  \fm_{\bar{F}}$.   An element $\iota$ of  $\DLr$ will be called \emph{trivial} provided that the second component of some, hence any, $(\beta,Z) \in \iota$ is trivial.
	\item The {\bf depth-$r$ Deligne-Lusztig parameter} of  $(x,X) \in \MP_r$ is defined to be  the equivalence class of $(\alpha,i_{\alpha,x}(X))$ in $\DLr$.  Here  $\alpha$ is any element of 
   $\bar{F}$ for which there exists a finite Galois $E/F$ such that $\alpha \in E$ and $(\alpha, E)$ is $(x,X)$-adapted.  This equivalence class will be denoted by  $\iota_x(X)$.  
	\item We say that two Moy-Prasad types of depth $r\,$
  are {\bf stable associates} provided that  they have the same depth-$r$ Deligne-Lusztig parameter.  That is, $(x,X), (x',X') \in \MP_r$ are stable associates if and only if $\iota_x(X) = \iota_{x'}(X')$.
    \end{enumerate}
\end{definition}

\begin{remark}
Definition \ref{def:DL} is a 
rational-depth generalization of the definition and construction in \cite[\S5.3 and \S5.4]{CB24}.    
\end{remark}

\begin{remark}
    By construction, we have that if $\iota_x(X)$ is nontrivial, then $(x,X)$ is nondegenerate.
\end{remark}

\subsection{Deligne-Lusztig parameters and  Moy-Prasad types} 

In this section we establish some connections between depth-$r$ Deligne-Lusztig parameters and Moy-Prasad types for $G(F)$ of depth $r$.

\subsubsection{Nondegeneracy and Deligne-Lusztig parameters in the tame case}  \label{subsub:whereNis}

Throughout Section~\ref{subsub:whereNis} we assume that our Moy-Prasad type has tame depth, that is, $r \in \Z_{(p)}$.

\begin{remark}\label{rmk:Jessica} If $p$ does not divide the order of the absolute Weyl group of $G$, then every nondegenerate Moy-Prasad type of depth $r$ has $r \in \Z_{(p)}$.    Indeed, if $(x,X)$ is a nondegenerate Moy-Prasad type of depth $r$, then from Frobenius reciprocity we know that $(x,X)$ occurs as a (nondegenerate) Moy-Prasad type for every irreducible quotient $\pi$ of $\ind_{G(F)_{x\geq r}}^{G(F)} \rho_X$.   Fix such a $\pi$.  When $p$ doesn't divide the order of the absolute Weyl group, from
\cite[Theorem~6.1]{Fin21} applied to the representation $\pi$ there exists a maximal tame $F$-torus $T$ such that $\ft^*(F)_{=-r} \neq \{0\}$. This implies $r\in\Z_{(p)}$.
\end{remark}

We begin this subsection by showing that a tame Moy-Prasad type  of depth $r$ is nondegenerate if and only if  its depth-$r$ Deligne-Lusztig parameter is nontrivial.  
To do this we use some results from~\cite[Section~3]{AD02}, which are stated for $\fg(F)$.  However, as noted in the introduction to Section~3 of {\it loc. cit.},  none of the results there rely on the structure of $\fg$  as a Lie algebra, and they are therefore valid for the Moy-Prasad filtration lattices of $\fg^*(F)$.

For  $s \in \R$ we define the $G$-domains $\fg^*(E)_{\ge s}$ and $\fg^*(E)_{> s}$ in $\fg^*(E)$ by
\[ \fg^*(E)_{\geq s}  = \bigcup_{y \in \cB(G,E)} \fg^*(E)_{y \geq s} \text{ \, and \,  } \fg^*(E)_{> s}  = \bigcup_{y \in \cB(G,E)} \fg^*(E)_{y > s}.\]  
Note that if  $E'/E$ is a tame
extension, then \(  \fg^*(E)_{\geq s} \subset \fg^*(E')_{\geq s}\)  and, similarly, \(  \fg^*(E)_{>s} \subset \fg^*(E')_{>s}\).
We know~\cite[Theorem~3.1.2]{AD02} that 
\[ \fg^*(E)_{\ge s}  = \bigcap_{z \in \cB(G,E)} \fg^*(E)_{z \geq s} + \NN_E^*  \text{ \, and \,  }  \fg^*(E)_{> s}  = \bigcap_{z \in \cB(G,E)} \fg^*(E)_{z > s} + \NN_E^*  \]
where $\NN^*_E$ is the set of nilpotent elements in $\fg^*(E)$; i.e., those $Y \in \fg^*(E)$ for which there exists $\lambda \in X_*^E(G)$ such that $\lim_{t \rightarrow 0} \Ad^*(\lambda(t)) Y = 0$.  We set $\NN^* = \NN^*_F$.

\begin{lemma}\label{lem:nonzeronewold} 
For this lemma and only this lemma we remove the assumption that $G$ splits over $F^t$.  
Suppose $(x,X) \in \MP_r$, and let $X_0 \in \fg^*(F)_{x \geq -r}$ be a lift of $X$.  The following statements are equivalent.
\begin{enumerate}
\item \label{item:a} $(x,X)$ is nondegenerate.
\item \label{item:b} The orbit of $X$ in $\fg^*(F^u)_{x=-r}$ under $G(F^u)_{x=0}$  does not contain zero in its closure.  Here we are thinking of  $\fg^*(F^u)_{x=-r}$ as a $k$-variety and $G(F^u)_{x=0}$ 
as an algebraic $k$-group.
\item \label{item:c} We have $X_0 \in \fg^*(F)_{\geq -r} \setminus \fg^*(F)_{>-r}$.
\item \label{item:c2} For any tame extension $L/F$ we have $X_0 \in \fg^*(L)_{\geq-r} \setminus \fg^*(L)_{>-r}$.
\end{enumerate}   
\end{lemma}

\begin{remark}  The analogue of the above result for Lie algebras is also true.  To see this, replace all occurrences of $\fg$ with $\fg^*$ in both its statement and proof.  
\end{remark}
  
\begin{proof}  Suppose  $(x,X) \in \MP_r$.   Let ${X}_0 \in \fg^*(F)_{x\geq -r}$ be a lift of $X$.

(\ref{item:a})$\Leftrightarrow$(\ref{item:b}):
This is~\cite[Prop. 13.7.4]{KP23} or \cite[Prop. 6.4]{MP94}.

(\ref{item:a})$\Rightarrow$(\ref{item:c}):
Since ${X}_0 \in \fg^*(F)_{x\geq -r} \subset \fg^*(F)_{\geq-r}$, we only need to show that ${X}_0 \not \in  \fg^*(F)_{>-r}$.   If ${X}_0 \in  \fg^*(F)_{>-r}$, then from~\cite[Theorem~3.1.2]{AD02} we have
\[X_0 \in \bigcap_{z \in \cB(G,E)} \fg^*(E)_{z > -r} + \NN_E^* \, \,  \subset \, \, \fg^*(E)_{x > -r} + \NN_E^*.\]  This contradicts our assumption that $(x,X)$ is nondegenerate.

(\ref{item:c})$\Rightarrow$(\ref{item:a}):  We will show the contrapositive.  If $(x,X)$ is degenerate, then
from~\cite[Corollary~3.2.6]{AD02} we have $X_0 \in \fg^*(F)_{>-r}$.

 (\ref{item:c})$\Leftarrow$(\ref{item:c2}):  This is immediate by taking $L$ to be $F$.

 (\ref{item:c})$\Rightarrow$(\ref{item:c2}):   This follows from ~\cite[Lemma~2.2.5]{AD04}; we provide a proof for the reader's convenience.
 
 It will be enough to show that for all tame extensions $L/F$ and all $s \in \R$ we have 
\[\fg^*(L)_{\geq s} \cap \fg^*(F) = \fg^*(F)_{\geq s}.\]

Fix $s \in \R$ and let $L/F$ be a tame extension.  Since $\fg^*(F)_{\ge s} \subset \fg^*(L)_{\ge s}$, it will be enough to show
\(\fg^*(L)_{\geq s} \cap \fg^*(F) \subset \fg^*(F)_{\geq s}\).
Let $L'/F$ be a tame Galois extension that contains $L$.  Since
$\fg^*(L)_{\geq s} \subset \fg^*(L')_{\geq s}$, it will be enough to show that 
\(\fg^*(L')_{\geq s} \cap \fg^*(F) \subset \fg^*(F)_{\geq s}\).

Suppose $Y \in \fg^*(L')_{\geq s} \cap \fg^*(F)$.  By definition, there exists $y \in \cB(G,L')$ such that $Y \in \fg^*(L')_{y\geq s}$.   Note that $Y \in  \fg^*(L')_{\tau (y) \geq s}$ for all $\tau \in \Gal(L'/F)$.  Thus, $Y \in  \fg^*(L')_{z \geq s}$ for all $z \in D$ where $D$ is the convex hull of $\{\tau(y) \colon \tau \in  \Gal(L'/F) \}$.  By the Bruhat-Tits fixed point theorem~\cite[Lemma~3.2.3]{BT72} or~\cite[Theorem~1.1.15(2)]{KP23}, there is a $z \in D$ that is $\Gal(L'/F)$-fixed.  Since $L'/F$ is tame, we have $z \in \cB(G,F)$.  Thus $Y \in \fg^*(L')_{z \geq s}^{\Gal(L'/F)} = \fg^*(F)_{z \geq s} \subset \fg^*(F)_{\geq s}$.
\end{proof}

\begin{lemma}\label{lem:nonzeronewnew} 
If $(x,X) \in \MP_r$ with $r \in \Z_{(p)}$, then
$(x,X)$ is nondegenerate if and only if
 $\iota_x(X)$ is nontrivial.
\end{lemma}

\begin{proof}  Suppose  $(x,X) \in \MP_r$ with $r \in \Z_{(p)}$.  Since both $\iota_x(X)$ and the notion of nondegeneracy are invariant under the action of $G(F)$, we may   assume that $x \in \cA$.   

Choose $\alpha \in F^t$ such that $\val(\alpha) = r$.  Let $E/F$ be a tame extension such that $\alpha \in E^\times$, $r \cdot e(E/F) \in \Z$,  and $G$ is $E$-split.  Since $u(E/F) = 0$, the pair $(\alpha,E)$ is $(x,X)$-adapted.   Let ${X}_0 \in \fg^*(F)_{x\geq -r}$ be a lift of $X$. 

From Lemma~\ref{lem:nonzeronewold} we know that $(x,X)$ is nondegenerate if and only if $X_0 \in \fg^*(E)_{\geq -r} \setminus \fg^*(E)_{> -r}$, which is equivalent to saying 
$\alpha X_0 \in \fg^*(E)_{\geq 0} \setminus \fg^*(E)_{> 0}$.
Using Lemma~\ref{lem:nonzeronewold} with $E$ playing the role of $F$ and $r$ being $0$, this is equivalent to saying
that  zero does not belong to the closure of the orbit $\Ad^*(G(E^u)_{x=0})(\overline{\alpha {X}_0})$ in $\fg^*(E^u)_{x=0}$. That is, it is equivalent to saying that $\iota_x(X)$ is nontrivial.
\end{proof}

\begin{remark}\label{rmk:Levy} Suppose $F=\bar{\F}_p((t))$ and let $E=\bar{\F}_p((t^{1/m}))$ where $m$ is a positive integer coprime to $p$. So $E/F$ is tame.  Denote by $\sigma\in\Gal(E/F)$ the element that sends $t^{1/m}$ to $\zeta_mt^{1/m}$ where $\zeta_m\in\bar{F}_p$ is a primitive $m$-th root of unity.  Let $H$ be a connected reductive group over $\bar{\F}_p$ with Lie algebra $\fh$.   Let $\theta$ be an order $m$ automorphism of $H$. We can base change $H$ to a connected reductive group $H_E$ over $E$.  Consider a descent datum of $H_E$ from $E$ to $F$ for which the twisted action of the inertia subgroup on $H_E(E)$ is given by $\sigma *h(t^{1/m})=\theta h(\zeta_mt^{1/m})$.  This gives a connected reductive group $\cH$ over $F$ with a distinguished vertex $x$ such that $\cH(F)_{x=0}=(H^{\theta})^o$ and $\cH(F)_{x=1/m}\cong\Lie\cH(F)_{x=1/m}=\fh^{\theta=\zeta_m}$. The analogue of Lemma \ref{lem:nonzeronewold} (but for the Lie algebra instead of the dual Lie algebra) then implies that an element $X\in\fh^{\theta=\zeta_m}$ is nilpotent in $\fh$ if and only if its $\Ad(H^{\theta})$-orbit contains $0$ in its closure. This gives another proof of \cite[Lemma 0.4, 2.11]{Lev09}.
\end{remark}

\begin{remark} If we assume that $x \in \cA$ is a hyperspecial vertex in $\cB(G,E)$ for some finite tame $E/F$, then we may identify $W$ with $W_x = N_{G(E^u)}(\bar{A})/\bar{A}$.   An appropriate variant of~\cite[Thm. 4.1]{RY14}, as in Remark~\ref{rmk:Levy}, shows that $\fg^*(F^u)_{x=-r}$ is an eigenspace in $\fg^*(E^u)_{x=0}$ under some finite cyclic group scheme action. This can be useful in computing $\iota_x(X)$ explicitly when $x$ is hyperspecial in $\cB(G,E)$ for $E/F$ tame.
\end{remark}

\subsubsection{Associativity and Deligne-Lusztig parameters}  
We no longer assume that $r \in \Z_{(p)}$. 
We now show that if two Moy-Prasad types are associate, then  the associated depth-$r$ Deligne-Lusztig parameters are stably associate.  

Recall that two positive-depth Moy-Prasad types $(x,X)$ and $(y,Y)$ of depth $r$ for $G(F)$ are {\bf associate} \cite[\S5.1]{MP94} provided that 
\begin{equation}
\Ad^*(G(F))(X_0+\fg^*(F)_{x>-r})\cap\Ad^*(G(F))(Y_0+\fg^*(F)_{y>-r})\not=\emptyset.   
\end{equation}
Here $X_0 \in \fg^*(F)_{x\geq-r}$ and $Y_0\in \fg^*(F)_{y\geq-r}$ are  lifts of $X$ and $Y$, respectively.
We now show

\begin{lemma}\label{lem:main1new} If two Moy-Prasad types for $G(F)$ of depth $r$ are associates, then they are stable associates.
\end{lemma}

\begin{proof} 
Given two Moy-Prasad types $(x,X)$ and $(y,Y)$ of depth $r$ that are associates, we will show  $\iota_x(X)=\iota_y(Y)$.     We begin with two  reductions.

First, by the definition of associativity there exists some $h \in G(F)$ such that $\Ad(h)^*X \cap Y  \neq \emptyset$.   Replacing $(x,X)$ with $h \cdot  (x,X)$ and choosing $X_0 \in X$ and $Y_0 \in Y$, we may   assume  $X \cap Y = (X_0 + \fg^*(F)_{x > r}) \cap (Y_0 + \fg^*(F)_{y > r}) \neq \emptyset$.

Second, we will show that we may assume that $r = 0$.  Choose a pair $(\alpha,L)$ that is $(x,X)$-adapted.  
   Note that
\[(\alpha X_0 + \fg^*(L)_{x > 0}) \cap (\alpha Y_0 + \fg^*(L)_{y > 0}) \neq \emptyset.\]
We also have  $i_{\alpha,x}(\bar{U}) = i^L_{1,x}(\overline{\alpha U})$ for all $U \in \fg^*(F)_{x \geq -r}$  and  $i_{\alpha,y}(\bar{V}) = i^L_{1,y}(\overline{\alpha V})$ for all $V \in \fg^*(F)_{y \geq -r}$.   To show that $\iota_x(X) = \iota_y(Y)$, we need to show that the equivalence classes represented by $(\alpha,i_{\alpha,x}(X)) = (1,i^L_{1,x}(\overline{\alpha X_0}) )$ and  $(\alpha,i_{\alpha,y}(Y)) = (1,i^L_{1,y}(\overline{\alpha Y_0}) )$ are equal.   For this, it will be enough to show $i^L_{1,x}(\overline{\alpha X_0})=i^L_{1,y}(\overline{\alpha Y_0})$.

Thus, without loss of generality we may replace $F$ with $L$  and assume 
\begin{itemize}
    \item $G$ is $F$-split,
    \item $X \in \fg^*(F)_{x=0}$ and $Y \in \fg^*(F)_{y=0}$,
    \item $X \cap Y = (X_0 + \fg^*(F)_{x>0}) \cap (Y_0 + \fg^*(F)_{y>0}) \neq \emptyset$.
\end{itemize}
We need to show that $i_{1,x}(X) = i_{1,y}(Y)$.

By \cite[Lemma 3.3]{KW76} or \cite[Prop. 3.3]{Spi21} when $p=2$, we have that $X\in\fg^*(F)_{x=0}$ is contained in $(\Lie B)^*$ for some Borel $\bar{k}$-subgroup $B\subset G(F^u)_{x=0}$, where $(\Lie B)^*\subset\fg^*(F^u)_{x=0}$ is the annihilator of $\Lie U_B$ with $U_B$ being the unipotent radical of $B$. Then $B$ corresponds to an alcove $C$ in $\cB(G,F^u)$ that contains $x$ in its closure. 
Since some apartment in $\cB(G(F^u))$ contains both $C$ and $y$ and all apartments in $\cB(G(F^u))$ are $G(F^u)$-conjugate, there is a $g \in G(F^u)$ such that $gC$ and $gy$ are contained in $\tilde{\cA}$, the apartment of $A^u$ in $\cB(G,F^u)$.    Note that  $i_{1,x}(X) = i^{F^u}_{1,gx}(gX)$ and $i_{1,y}(Y) = i^{F^u}_{1,gy}(gY)$, To ease notation, we conjugate our entire picture by $g$. 
We have $y \in \tilde{\cA}$, $x \in \bar{C} \subset \tilde{\cA}$, $\bar{A} \subset B$, and   $(X_0 + \fg^*(F^u)_{x > 0}) \cap (Y_0 + \fg^*(F^u)_{y > 0}) \neq \emptyset$.    We need to show that $i^{F^u}_{1,x}(X) = i^{F^u}_{1,y}(Y)$

Let $x'$ be any point in (the interior of) $C$. 
We may write $X=X'+X_u$ where $X'\in\bar{\fa}^*$ and $X_u\in(\Lie U_B)^*$, the latter being the annihilator of $\Lie B$. Since $\bar{A}$ acts on $X_u$ with anti-dominant weights (with respect to $B$), we have 
\[
i_{\alpha,x}(X)=i_{\alpha,x}(X')
\text{ and }
X_0+\fg^*(F^u)_{x>0}\subset X_0'+\fg^*(F^u)_{x'>0}
\]
 where $X_0, X_0' \in \fg^*(F^u)_{x \geq 0}$ are lifts of $X$ and $X'$, respectively.   Note that  
\[
\emptyset \neq (X_0 + \fg^*(F^u)_{x > 0}) \cap (Y_0 + \fg^*(F^u)_{y > 0}) \subset (X'_0 + \fg^*(F^u)_{x' > 0}) \cap (Y_0 + \fg^*(F^u)_{y > 0}).
\]
Since $x'$ is in (the interior of) an alcove, the map $\bar{\fa}^*\ira\fg^*(F^u)_{x'=0}$ is an isomorphism. 
From the  second to last row of \eqref{eq:squaresnew} we conclude that 
\begin{equation}   \label{equ:reductiontoC}
    i^{F^u}_{1,x}(X)=i^{F^u}_{1,x}(X')=i^{F^u}_{1,x'}(X').
\end{equation} 
We replace $(x,X)$ with $(x', X')$. 

Repeating the above argument for $(y,Y)$, we are reduced to the case when $x$ and $y$ are in (the interior of) some  alcoves in $\tilde{\cA}$ and 
\begin{equation} \label{equ:assocnew}
 (X_0 + \fg^*(F^u)_{x > 0}) \cap (Y_0 + \fg^*(F^u)_{y > 0}) \neq \emptyset.
\end{equation}
If the alcoves are the same, then we may take $x = y$.  Then~\eqref{equ:assocnew} shows that $X = Y$, and we are done.   So, assume that the alcoves are not the same.

For any $w$ on the geodesic $[x,y]$ from $x$ to $y$ in $\tilde{\cA}$, we have
\[
\fg^*(F^u)_{x\ge 0}\cap\fg^*(F^u)_{y\ge 0}\subset\fg^*(F^u)_{w\ge 0}
\]
and so from~\eqref{equ:assocnew} there exists $W_0 \in\fg^*(F^u)_{w\geq 0}$ such that both
\[
(X_0+\fg^*(F^u)_{x>0})\cap(W_0+\fg^*(F^u)_{w>0})\not=\emptyset \]
and
\[ (W_0+\fg^*(F^u)_{w>0})\cap(Y_0+\fg^*(F^u)_{y>0})\neq \emptyset.
\]
This allows us to reduce to the case when the alcove containing $x$ and the alcove containing $y$ have a common point $z \in [x,y]$ in their closures, i.e. to the case when the images of  $G(F^u)_{x\geq 0}$ and $G(F^u)_{y \geq 0}$ in $G(F^u)_{z=0}$ are two Borel subgroups $B_x,B_y\subset G(F^u)_{z=0}$. The nonempty intersection of~\eqref{equ:assocnew}
implies that
\[ \left(\bar{X}_0+(\Lie U_{B_x})^*\right)\cap\left(\bar{Y}_0+(\Lie U_{B_y})^*\right)\neq \emptyset.\]
Here $\bar{X}_0 \in \fg^*(F^u)_{z=0}$ is the image of  a  lift $X_0 \in \fg^*_{x \geq 0} \subset \fg^*_{z \geq 0}$ of  $X \in \fg^*_{x = 0}$; $\bar{Y}_0$ is defined similarly.
 If  $Z \in \fg^*(F^u)_{z=0}$ belongs to this intersection, then the same argument that produced~\eqref{equ:reductiontoC} shows that  $i^{F^u}_{1,x}(X)=i^{F^u}_{1,z}(Z)=i^{F^u}_{1,y}(Y)$. Hence we are done.
\end{proof}

\begin{corollary}  \label{cor:iotapi} Suppose $(\pi,V)$ is an irreducible representation of $G(F)$ of positive depth.  We can naturally associate to $\pi$ an $\iota_\pi$ in $\DLpi$.  If $\rho(\pi) \in \Z_{(p)}$, then this $\iota_\pi$ is  nontrivial.  
\end{corollary}
   
\begin{proof}  Suppose $(\pi,V)$ is a positive depth  irreducible representation of $G(F)$. 
Moy and Prasad proved \cite[proof of Thm. 5.2 in Section~7]{MP94} that nondegenerate Moy-Prasad types occur in $\pi$. Moreover, any two nondegenerate Moy-Prasad types that occur in $\pi$ (i) are  of the same depth and (ii) are associate.  Lemma~\ref{lem:main1new} shows that we then have a well-defined  Deligne-Lusztig parameter $\iota_\pi := \iota_y(Y) \in\DLpi$ attached to $\pi$.  If $\rho(\pi) \in \Z_{(p)}$, then $\iota_\pi$ is nontrivial from Lemma~\ref{lem:nonzeronewnew}.
\end{proof}

\begin{remark}
In the case when $G$ is a simply-connected $F$-split
group and $\rho(\pi)$ is integral, versions of
Lemma~\ref{lem:nonzeronewnew} and Lemma~\ref{lem:main1new} were proved in \cite[Proposition 5.9]{CB24} via a different (and indirect)
method.
\end{remark}

\section{Restricted depth-\texorpdfstring{$r$}{r} parameters}  \label{sec:C2Bnew}
Our next step is to construct a restricted Langlands parameter for any $\iota \in \DLr$ with $r > 0$.   

\subsection{A construction of restricted toral homomorphisms arising from Deligne-Lusztig parameters}  \label{sec:restrictedtoralhoms}
Fix $\iota \in \DLr$ with $r > 0$.  

Our first goal is to associate to $\iota$ a continuous homomorphism from $\Group_F^{r:r+}$ to $A^\vee$.

\begin{definition}
    Denote by $i_{\alpha,E}^A:\fa^*(E^u)_{=-r}\sra\bar{\fa}^*\quo W$ the composition of the last two rows of \eqref{eq:squaresnew}.
\end{definition}

\begin{lemma} \label{lem:iotatoMP}
 There exist a finite Galois extension $E/F$, an element $\beta \in E$ with $\val(\beta) = r$, and $X \in \fa^*(E)_{=-r}$ such that  $E$ is $(X,r,A)$-adapted and $(\beta,i_{\beta,E}^A(X)) \in \iota$.  
\end{lemma}

\begin{proof}  Suppose $(\alpha,Y) \in \iota$.
Thanks to Lemma~\ref{lem:Eexists} we can choose a finite Galois extension $E$ of $F$ such that $u(E/F) < r$ and  $e(E/F) \cdot r \in \Z$.  Since $A$ splits over a tame extension, by Lemma~\ref{lem:Eexists} we may also assume that $A$ is $E$-split.   Choose $\beta \in E$ with $\val(\beta) = r$.  
Let $c \in \bar{k}^\times$ denote the image of $\alpha/\beta$ in $\cO^\times_{\bar{F}}/(1+\fm_{\bar{F}})$.  Without loss of generality, we may replace $E$ with a finite unramified extension of itself such that $c \in k_E$.  Replace $\beta$ with $c_0 \beta \in E$ where  $c_0 \in \cO_E$ is a lift of $c$.     Choose $X \in \fa^*(E^u)_{=-r}$ with $i_{\beta,E}^A(X)=Y$.   By replacing $E$ with a finite unramified extension of itself, we may assume $X \in \fa^*(E)_{=-r}$.
  Note that  $E$ is $(X,r,A)$-adapted and $(\beta,i_{\beta,E}^A(X)) \in \iota$.
  \end{proof}

Choose $E$, $\beta$, and $X$ as in the statement of Lemma~\ref{lem:iotatoMP}.
Let $\phi_{X,E}^A \colon \Group_F^{r:r+} \ra G^{\vee}$ denote the composition of $A^\vee \ira G^\vee$ and the continuous homomorphism $\Group_F^{r:r+} \ra A^{\vee}$ constructed in Section~\ref{sec:constructphi}.  Thanks to Corollary~\ref{cor:base} we know that $\phi_{X,E}^A$ is independent of the choice of $E$, and so we may write $\phi_{X}^A$.

\begin{lemma} \label{lem:welldefinedoniota}
The $G^\vee$-conjugacy class of $\phi_X^A$ depends only on $\iota$.
\end{lemma}

Thanks to Lemma~\ref{lem:welldefinedoniota} we have a map  from $\DLr$ to the set of $G^\vee$-conjugacy classes of tame restricted depth-$r$ parameters that is given by sending $\iota \in \DLr$ to the $G^\vee$-conjugacy class of $\phi^A_X$.  We abuse notation and write $\varphi_\iota$ for an element of the $G^\vee$-conjugacy class of $\phi^A_X$.

\begin{proof} 
 Suppose we have a Galois extension $E'/F$, a $\beta' \in E'$ with $\val(\beta') = r$, and $X' \in \fa^*(E')_{=-r}$ such that  $E'$ is $(X',r,A)$-adapted and $(\beta',i_{\beta',E}^A(X')) \in \iota$.  
We will show that $\phi_{X'}^A$ is $G^\vee$-conjugate to $\phi_{X}^A$.

From Corollary~\ref{cor:base} we know that the construction of  $\phi_{X'}^A$ (resp. $\phi_X^A$) is independent of the choice of the auxiliary field $E'$ (resp. $E$).  Using Lemma~\ref{lem:adatptedreduction} we may assume $E'=E$ is an extension of $F$ such that $\beta, \beta' \in E^\times$ and $E$ is both $(X,r,A)$- and $(X',r,A)$-adapted.

Set $Y = i_{\beta,E}^A(X)$ and $Y' = i_{\beta',E}^A(X')$.
 Since $(\beta,Y), (\beta',Y') \in \iota$, we have $Y'=j_c(Y)$ where $c = \beta'/\beta +\fm_{\bar{F}}$.  It follows that  $i^A_{\beta',E}(X) = Y'$.   Thus, we are reduced to showing that if  $X',X \in \fa^*(E^{ u})_{x=-r}$ such that $i_{\beta',E}^A(X')=i_{\beta',E}^A(X)$, then $\phi_{X'}^A$ is $G^\vee$-conjugate to $\phi_{X}^A$.
 Since $i_{\beta',E}^A(X')=i_{\beta',E}^A(X)$, the elements  $X$ and $X'$ are in the same $W = N_G(A)/A$-orbit. This implies that $\phi^A_{X},\phi^A_{X'}:\Group_F^{r:r+}\ra A^{\vee} \ira G^\vee$ are in the same $N_{G^{\vee}}(A^{\vee})/A^{\vee}$-orbit, hence they are  $G^{\vee}$-conjugate.
\end{proof}
 
Recall that we write $\varphi_\iota$ for an element of the $G^\vee$-conjugacy class corresponding to $\iota \in \DLr$.

\begin{corollary}
\label{cor:comparenew}  Suppose $T$ is a maximal $F^t$-split torus in $G$ and $(x,X) \in \MP_r$ is $T$-toral as in Section~\ref{sec:Ttoraltype}. 
    The homomorphism $\varphi_{\iota_x(X)} \colon \Group_F^{r:r+} \ra G^{\vee}$ constructed above  is $G^{\vee}$-conjugate to the homomorphism $\phi^T_{X} \colon \Group_F^{r:r+}\ra T^{\vee}$
    that was constructed in Section~\ref{sec:constructphi}.
\end{corollary}
 
\begin{proof}
Since $(x,X) \in \MP_r$ is $T$-toral, there exists  a finite Galois extension $E$ of $F$  such that 
$T$ is defined over $E$,
$E$ is $(X,r,T)$-adapted, and
$x$ belongs to the building of $T(E^t)$ in $\cB(G(E^t))$.

 In Section~\ref{sec:constructphi} we constructed a homomorphism $\phi^T_{X,E}$ from this data.   Since $E$ is $(X,r,T)$-adapted, we can choose  $\alpha \in E^\times$ with $\val(\alpha) = r$.

Since $A$ and $T$ are both $E$-split, there exists $g \in G(E)$ such that $A_E = \Int(g)  T$.  Since $A_E$ and $T$ are isomorphic as $E$-tori,  by construction we have that  $\phi^T_{X,E}$  and $\phi^A_{\Ad(g)^*X,E}$  lie in the same $G^\vee$-conjugacy class. 
Thus, it will be enough to show that $\phi^A_{\Ad(g)^*X,E}$ and $\varphi_{\iota_x(X)}$ are $G^\vee$-conjugate.

The pair $(\alpha, i_{\alpha,x}(X))$ belongs to the equivalence class $\iota_x(X)$.   Since 
\[  i_{\alpha,x}(X) =  i^E_{\alpha,x}(X) = i^E_{\alpha,gx}(\Ad^*(g)X),\]
we conclude that $(\alpha, \Ad^*(g)X) $ also belongs to $\iota_x(X)$.    The result now follows from 
Lemma~\ref{lem:welldefinedoniota}.
\end{proof}

\subsection{Restricted depth-$r$ parameters and Deligne-Lusztig parameters}

\begin{definition}\label{def:TRP} Suppose $r > 0$.  A continuous homomorphism $\varphi:\Group_F^{r:r+}\ra G^{\vee}$ or, equivalently, a continuous homomorphism $\varphi:\Group_F^r \ra G^{\vee}$ that is trivial on $\Group_F^{r+}$, is called a 
 {\bf restricted depth-$r$ parameter}.  We say $\varphi$ is {\bf tame} provided that
\begin{enumerate}
    \item $\im(\varphi)$ is contained in a (maximal) torus of $G^\vee$, and
    \item $\varphi$ is trivial on $\Group_F^r\cap(\Group_F^{0+},\Group_F^{0+})$, where $(\Group_F^{0+},\Group_F^{0+})$ is the closure of the commutator of $\Group_F^{0+}$.
\end{enumerate}
We denote by $\TPr$ the set of $G^{\vee}$-conjugacy classes of tame restricted depth-$r$ parameters.
\end{definition}

\begin{remark}\label{rmk:tame-eqiv}  For $r  > 0$, the sets $\TPr$ defined in the Introduction and defined in Definition~\ref{def:TRP} are the same. Indeed, if $\varphi\equiv\til{\varphi}|_{\Group_F^r}$ for some $\til{\varphi}:\Group_F^{0+}\ra T^{\vee}$, then $\im(\varphi)$ is contained in a maximal torus and $\varphi$ is trivial on $\Group_F^r\cap (\Group_F^{0+},\Group_F^{0+})$. Conversely, if $\varphi$ is contained in a maximal torus $T^{\vee}\subset G^{\vee}$ and trivial on both $\Group_F^r\cap (\Group_F^{0+},\Group_F^{0+})$ and $\Group_F^{r+}$, then it can be viewed as a homomorphism $\varphi:\Group_F^{r}/\Group_F^{r+}(\Group_F^r\cap (\Group_F^{0+},\Group_F^{0+}))\ra T^{\vee}$. The quotient group $\Group_F^{r}/\Group_F^{r+}(\Group_F^r\cap (\Group_F^{0+},\Group_F^{0+}))$ is a closed subgroup of the abelian pro-finite topological group $\Group_F^{0+}/\Group_F^{r+}(\Group_F^{0+},\Group_F^{0+})$. Hence there exists $\til{\varphi}:\Group_F^{0+}/\Group_F^{r+}(\Group_F^{0+},\Group_F^{0+})\ra T^{\vee}$ that restricts to $\varphi$. This shows that  the definition given in the Introduction and Definition~\ref{def:TRP} are equivalent.
\end{remark}

As discussed in Section~\ref{sec:restrictedtoralhoms},  for $\iota \in \DLr$ the homomorphism $\varphi_{\iota}$ is conjugate to $\phi_X^A$ for some $X$, and hence it is the restriction to $\Group_F^r$ of an $L$-homomorphism $W_F \rightarrow  A^\vee$.   Thus, the $G^\vee$-conjugacy class of $\varphi_{\iota}$  belongs to $\TPr$.   The following lemma shows that the converse is true as well.

\begin{lemma}\label{lem:main2new}    Suppose $G$ is a connected reductive $F$-group that splits over $F^t$ and $r\in\Q_{>0}$. The map $\iota \mapsto \varphi_{\iota}$ induces a bijection between $\DLr$ and $\TPr$.
\end{lemma}

\begin{proof}
We first show that the map $\iota \mapsto \varphi_{\iota}$ induces an injection from $\DLr$ to $\TPr$.  Suppose $\iota_j \in \DLr$ for $j \in \{1,2\}$ such that $\varphi_{\iota_1}$ is $G^\vee$-conjugate to $\varphi_{\iota_2}$.

Thanks to Lemma~\ref{lem:iotatoMP} for $j \in \{1,2\}$ 
 there exist a finite Galois extension $E_j/F$, a $\beta_j \in E_j$ with $\val(\beta_j) = r$, and $X_j \in \fa^*(E_j)_{=-r}$ such that  $E_j$ is $(X_j,r,A)$-adapted and $(\beta_j,\iota_{\beta_j,E}^A(X_j)) \in \iota_j$.    Using Lemma~\ref{lem:adatptedreduction} we may assume $E = E_1 = E_2$.
Since $(\beta_2,\iota_{\beta_2,E}^A(X_2)) \equiv (\beta_1,j_c(\iota_{\beta_2,E}^A(X_2)))$ where $c = \beta_2/\beta_1  +\fm_{\bar{F}}$,  we may assume $\beta_1 = \beta_2$.   Set $\beta = \beta_1 = \beta_2$.  We now need to show that $(\beta,\iota_{\beta,E}^A(X_1)) \equiv (\beta,\iota_{\beta,E}^A(X_2))$.

Recall that $\varphi_{\iota_j} = \phi_{X_j}^A$, so there exists $\gamma \in G^\vee$ such that 
$\phi_{X_2}^A = {}^{\gamma}(\phi_{X_1}^A)$.

 Let $M^{\vee}_j:=C_{G^{\vee}}(\phi_{X_j}^A)^o$. We then have  $M^{\vee}_2={}^{\gamma}M^{\vee}_1$. Since  $M^{\vee}_1$ and $M^{\vee}_2$ contain $A^{\vee}$, 
 there exists $m_2 \in M^{\vee}_2$ such that $m_2 \gamma$ normalizes $A^\vee$.  Replacing $\gamma$ with $m_2 \gamma$, we have $M^{\vee}_2={}^{\gamma}M^{\vee}_1$ and $\gamma$ normalizes $A^{\vee}$.   Denote by $w$ the image of $\gamma$ in $N_{G^{\vee}}(A^{\vee})/A^{\vee}$, so that $\phi^A_{X_2} ={}^{w}\phi^A_{X_1}$.  This implies that $X_2$ belongs to the $W$-orbit of $X_1$ in $\fa^*(E^u)_{x=r}$,
which implies 
$\iota_{\beta,E}^A(X_1) = \iota_{\beta,E}^A(X_2)$ in $\bar{\fa}^*\quo W$

We now show that the induced map from $\DLr$ to $\TPr$ is surjective. We begin with a continuous homomorphism $\varphi:\Group_F^{r:r+}\ra A^{\vee}$ that is a tame restricted depth-$r$ parameter as in the Introduction.  Let $\til{\varphi}:\Group_F^{0+}/\Group_F^{r+}\ra A^{\vee}$ be an extension of $\varphi$. Let $U\subset \Group_F^{0+}$ be an open subgroup that is contained in the kernel of $\til{\varphi}$. Since $\Group_F^{0+}$ is compact, there exists an open subgroup $V\subset\Gal(\bar{F}/F)$ such that $\sigma\tau\sigma^{-1}\tau^{-1}\in U$ for any $\sigma\in V$, $\tau\in \Group_F^{0+}$. In particular, $\til{\varphi}$ is invariant under conjugation by $V$, and thus it extends to $\til{\varphi}':V\Group_F^{0+}\ra A^{\vee}$. We can find $L/F$ finite tame such that $A$ is $L$-split and  $W_L\subset\Gal(\bar{F}/L)\subset V\cdot \Group_F^{0+}$. 

By local Langlands for tame tori and its preservation of depth, we know that $\til{\varphi}'|_{W_L}$ corresponds to a continuous character $\chi:A(L)/A(L)_{>r}\ra\Cc$. The restriction $\chi|_{A(L)_{=r}}$ then gives a character on $\fa(L)_{=r}$ necessarily of the form
\[
W\mapsto\Lambda_L(W,X)
\]
for some $X\in \fa^*(L)_{=-r} \subset \fa^*(\bar{F})_{=-r}$. 

Using Lemma~\ref{lem:lemma12} we can find a Galois extension $E/F$ such that $E$ is $(X,r,A)$-adapted.   Choose $\alpha \in E$ such that $\val(\alpha) = r$.  Let $Y:=\iota_{\alpha,E}^A(X)$. Then $(\alpha,Y)\in\DLr$ is mapped to $\varphi\in\TPr$. 
\end{proof}

\subsection{A proof of Theorem~\ref{thm:mainthm} when $r>0$}

\begin{lemma}  \label{lem:pfinpos}
Theorem~\ref{thm:mainthm} is valid in the positive depth situation.  That is, for all $r>0$ such that $\Irr(G(F))_r \neq \emptyset$ there is a map 
\[
\Irr(G(F))_r \ra \TPr. 
\]
Moreover, if we also have $r \in \Z_{(p)}$, then every element of the image of this map is nontrivial. 
\end{lemma}

\begin{proof}
The existence of the map follows from Corollary~\ref{cor:iotapi} and Lemma~\ref{lem:main2new}.
The final statement follows from Corollary~\ref{cor:iotapi}.
\end{proof}

\section{Some conjectures} 
\label{sec:conjnew}

Recall that $G$ is a connected reductive $F$-group that splits over $F^t$.  We may define ${}^LG:=(W_F/\Group_F^{0+})\ltimes G^{\vee}$. If $\varphi:W'_F :=  W_F \times \SL_2 \ra{}^LG$  is a Langlands parameter for $G(F)$, then we denote by $\Pi(\varphi)$ the associated $L$-packet of irreducible representations on $G(F)$.

\subsection{The main conjectures}
Thanks to Corollary~\ref{cor:iotapi} and Lemma~\ref{lem:main2new}, 
we can now generalize Conjecture \ref{conj:desi-pre} to the following desideratum for the local Langlands correspondence.

\begin{conjecture}\label{conj:desinew} Recall that $G$ splits over $F^t$. Suppose that an irreducible smooth representation $\pi$ of $G(F)$ has a nondegenerate  Moy-Prasad type $(x,X)$ of depth $r>0$. Then the restriction of a Langlands parameter for $\pi$ to $\Group_F^r$ is $G^{\vee}$-conjugate to $\varphi_{\iota_\pi} = \varphi_{\iota_x(X)}$. In particular, the restriction is tame (see Definition \ref{def:TRP}).
\end{conjecture}

\begin{remark} \label{rem:indofaddchar2} We note that Conjecture \ref{conj:desinew} contains Conjecture \ref{conj:desi-pre} thanks to Corollary~\ref{cor:comparenew}.  Also note that as with Conjecture~\ref{conj:desi-pre}, the  validity of Conjecture \ref{conj:desinew} is independent of the choice of $\Lambda_F$ (see Remark~\ref{rem:indofaddchar}). 
\end{remark}

From Lemma~\ref{lem:main2new} we know that a nontrivial tame restricted depth-$r$ parameter can be associated to a nontrivial depth-$r$ Deligne-Lusztig parameter.  
The construction in Section~\ref{constructionofDL} shows that any Moy-Prasad type that is associated to a nontrivial depth-$r$ Deligne-Lusztig parameter must be nondegenerate. 
 This suggests a conjecture in the opposite direction of Conjecture~\ref{conj:desinew}.

 \begin{conjecture}\label{conj:desi2}  
Recall that $G$ splits over $F^t$.
Suppose that $\varphi:W_F' \ra{}^LG$  is a Langlands parameter and
$r>0$.   If  the restriction $\varphi|_{\Group_F^r}$ of $\varphi$ to $\Group_F^r$ is a nontrivial tame restricted depth-$r$ Langlands parameter for $G(F)$, then $\varphi_{\iota_x(X)}$ is $G^\vee$-conjugate to 
$\varphi|_{\Group_F^r}$ for every  Moy-Prasad type $(x,X)$ of depth $r$ occurring in a representation  in $\Pi(\varphi)$.
\end{conjecture}

\subsection{Depth and the main conjectures}
On the representation theory side the {depth}, $\rho(\pi)$, of an irreducible smooth representation $(\pi,V)$ of $G(F)$ measures the first occurrence of fixed vectors with respect to the Moy-Prasad subgroups, and it is defined~\cite[\S5.2]{debacker:someapplicationsgroup}  by
\[ \rho(\pi) = \min \{s \in \Q_{\geq 0} \, | \, \text{there exists
$y \in \cB(G(F))$ such that $V^{G_{y>s}} \neq \{0\}$}  \}.\]

\begin{remark}
    If $(\pi,V)$ admits a nondegenerate Moy-Prasad type $(x,X)$, then from~\cite[Thm. 5.2]{MP94}, \cite[Thm. 3.5]{MP96} we have $(x,X) \in \MP_{\!\rho(\pi)}$. 
\end{remark}

On the Galois side the {depth}, $\rho(\varphi)$, of a Langlands parameter $\varphi:W_F' \ra{}^LG$ measures, with respect to the upper-numbering filtration,  the smallest quotient  
through which $\varphi$ factors, and it is defined by
\[ \rho(\varphi) = \min  \{s \in \Q_{\geq 0} \, | \, \text{the restriction of $\varphi$ to ${\Group_F^{s+}}$ is trivial}  \}.\]
If $G$ is not assumed to split over a tame extension, then this definition needs to be modified (see, for example,~\cite[Definition~2.12 and Lemma~2.14]{AP22}).

\begin{remark}  If $\rho(\varphi)>0$, then  $ \varphi|_{\Group_F^{\rho(\varphi)}}$ is nontrivial.
\end{remark}

Suppose $\varphi$ is a positive depth Langlands parameter and $\pi \in \Pi(\varphi)$.  The relationship between $\rho(\pi)$ and $\rho(\varphi)$ for groups that split over $F^t$ has been explored in many contexts.
For example, if $G$ is an inner form of $\mathrm{GL}_n$, then from~\cite[Theorem~2.9]{ABPS16} we have  $ \rho(\varphi) = \rho(\pi)$. This equality of depths also holds for $\GSP_4$ when $p$ is odd~\cite{Gan15} and unitary groups when $p$ is sufficiently large~\cite{Oi21,Oi23}.
If $G$ is an
inner form of $\mathrm{SL}_n$, then from~\cite[Corollary~3.4 and Theorem~3.8]{ABPS16} we have $\rho(\varphi) \leq \rho(\pi)$ with equality if $\varphi[\Group_F^{0+}]$ lies in a maximal torus. 
Building on~\cite{GV17}, it is shown in~\cite{Oi23} that 
if $G$ is a quasi-split symplectic or special orthogonal group and $p$ is sufficiently large, then $\rho(\pi) \leq \rho(\varphi)$ with equality  achieved for at least one $\pi \in \Pi(\varphi)$.       Finally, there are examples where the inequality is strict: (a) if $F$ has characteristic two, then for inner forms of $\mathrm{SL}_2$ we  have $ \rho(\varphi) < \rho(\pi)$ if and only if $\varphi[\Group_F^{0+}]$ does not lie in a maximal torus~\cite{AMPS17};  (b) for $\mathrm{SL}_2$ over $\Q_2$, Mark Reeder constructed an example where $\rho(\varphi) = 1/3 < 1/2 = \rho(\pi)$ (see~\cite[Example 3.5]{ABPS16}); and (c) there are examples where the inequality is strict for the  groups $\SU_p(\Q_p)$ and $G_2(\Q_2)$~\cite[Sections~7.4 and 7.5]{RY14}.
These examples suggest that the following guess may be warranted.

\begin{conjecture} \label{conj:desi2newnew} Recall that $G$ splits over $F^t$.
Suppose $\varphi$ is a positive depth Langlands parameter for $G(F)$.  For all $\pi \in \Pi(\varphi)$ we have $\rho(\varphi)\le\rho(\pi)$. If moreover that $p$ does not divide the order of the absolute Weyl group of $G$, then $\rho(\varphi)=\rho(\pi)$.
\end{conjecture}

\begin{lemma} \label{lem:45to50}  If Conjecture \ref{conj:desinew} is true, then Conjecture \ref{conj:desi2newnew} is true.
\end{lemma}

\begin{proof}   Suppose $\varphi$ is a positive depth Langlands parameter and $\pi \in \Pi(\varphi)$. 
From Conjecture~\ref{conj:desinew} we have that $\varphi|_{\Group_F^{\rho(\pi)}}$ is given by the construction in \S\ref{sec:C2Bnew}.  In particular, $\varphi|_{\Group_F^{\rho(\pi)}}$  is  trivial on $\Group_F^{\rho(\pi)+}$. Hence $\rho(\varphi)\le\rho(\pi)$. The last assertion follows from Remark \ref{rmk:Jessica} and Lemma \ref{lem:pfinpos}.
\end{proof}

\begin{corollary} If Conjecture \ref{conj:desinew} is true, then Conjecture \ref{conj:desi2} is true.
\end{corollary}

\begin{proof}
Suppose $\varphi$ is a positive depth Langlands parameter and $\pi \in \Pi(\varphi)$. 
Let $r=\rho(\varphi)$. Suppose $\varphi|_{\Group_F^r}$ is a nontrivial tame restricted depth-$r$ Langlands parameter for $G(F)$.   The depth of $\varphi$ is thus $r$, and from Lemma~\ref{lem:45to50} we have  $\rho(\pi) \geq \rho(\varphi) = r$ for all $\pi \in \Pi(\varphi)$. If $\rho(\pi)>r$, then there are no nondegenerate Moy-Prasad types of depth $r$ in $\pi$, and so the conclusion of Conjecture~\ref{conj:desi2}  is vacuously true.
If $\rho(\pi)=r$, then let  $(x,X)$ be a nondegenerate Moy-Prasad type of depth $r$ in $\pi$. 
Conjecture~\ref{conj:desinew}
says that the restriction of $\varphi$ to $\Group_F^{r}$ is $G^{\vee}$-conjugate to $\varphi_{\iota_\pi} = \varphi_{\iota_x(X)}$, and the conclusion of   Conjecture \ref{conj:desi2} therefore holds.
\end{proof}

\begin{remark}\label{rmk:counter}
It is natural to ask if Conjecture \ref{conj:desi2} could be upgraded to ``If  the restriction $\varphi|_{\Group_F^{r}}$ of $\varphi$ to $\Group_F^{r}$  is a nontrivial  tame restricted depth-$r$ Langlands parameter for $G(F)$, then any $\pi\in\Pi(\varphi)$ has a nondegenerate Moy-Prasad type $(x,X)$ of depth $r$ such that $\varphi_{\iota_x(X)}$ is $G^\vee$-conjugate to 
$\varphi|_{\Group_F^r}$.'' This, however, is false as we now show.

 Let $p$ be any prime, $G=\mathrm{SL}_p$, and $F=\F_p((t))$. (A similar construction should be possible for any $p$-adic field.) Let $S=\{n\in\Z\;|\;n\ge 3\}$ if $p>2$ and $S=\{2\}\sqcup\{n\in\Z\;|\;n\ge 4\}$ if $p=2$. Let $U$ be the subgroup of $F^{\times}$ generated by
\[
U=\langle t,\F_p^{\times},(1+t^n)_{n\in S}\rangle.
\]
By local class field theory, $U$ corresponds to a totally ramified abelian extension $E/F$ with $\Gal(E/F)\cong F^{\times}/U$, which is a product of two cyclic groups of order $p$. We have a particularly interesting embedding $\varphi:\Gal(E/F)\ira\mathrm{PGL}_p(\C)$ whose image is the abelian group of order $p^2$ generated by $\mathrm{diag}(1,\zeta_p,\zeta_p^2,...,\zeta_p^{p-1})$ and $\matr{0&1&0&...&0\\0&0&1&...&0\\&&&...&\\0&0&0&...&1\\1&0&0&...&0}$. In particular $\varphi$ is trivial on the commutator of $W_F$  while $\im(\varphi)$ is not contained in a maximal torus.

Write $d=2$ if $p>2$ and $d=3$ if $p=2$. Then $\varphi$ has depth $d$.  However, the upper numbering subgroup $\Group(E/F)^d$ is given by the subgroup of $F^{\times}/U$ of depth $d$, which is a cyclic group of order $p$. In particular, $\varphi(\Group(E/F)^d)$ { is contained in a maximal torus} and thus $\varphi|_{\Group_F^d}$ is a tame restricted depth-$d$ parameter.

If every $\pi \in \Pi(\varphi)$ has a nondegenerate Moy-Prasad type $(x,X)$ of depth $r$, then $\rho(\pi) = \rho(\varphi) = r$ for all $\pi \in \Pi(\varphi)$. This seems very unlikely because of \cite[Theorem~1.1]{AMPS17}, which confirms that this is not the case when $p=2$.

To take this a bit further, let $U'$ be the subgroup generated by $U$ and $1+t$. Then $F^{\times}/U'$ is cyclic of order $p$. It corresponds to some intermediate extension $F\subset E'\subset E$, with $\Gal(E'/F)=\Group(E'/F)^d$, and the natural surjection $\Group(E/F)^d\sra\Gal(E'/F)$ is an isomorphism. One has a parameter $\varphi':\Gal(E'/F)\ira\mathrm{PGL}_p(\C)$ with image contained in a maximal torus.  In this case, $\varphi$ and $\varphi'$ both have depth $d$ and have the same restriction to $\Group_F^d$, but we expect that $\rho(\pi)>d$ for all $\pi \in \Pi(\varphi)$ and $\rho(\pi')=d$ for all $\pi' \in \Pi(\varphi')$; the latter is again confirmed when $p=2$ by \cite[Theorem~1.1]{AMPS17}.
\end{remark}

\begin{remark} Conjecture~\ref{conj:desinew} and its consequences do not rule out the following scenario. There could exist a positive depth Langlands parameter $\varphi$, possibly with $\varphi|_{\Group_F^{\rho(\varphi)}}$ tame, such that (a) at least two depths occur among the representations in $\Pi(\varphi)$ and (b) for all $ \pi \in \Pi(\varphi)$ we have $\rho(\pi) \geq \rho(\varphi)$ with $\rho(\pi) > \rho(\varphi)$ if and only if $\varphi_{\iota_\pi}$ is trivial.
 If Conjecture~\ref{conj:desinew} holds, then from Corollary~\ref{cor:iotapi} we know that we can only have $\rho(\pi) > \rho(\varphi)$
when  $\rho(\pi)\not\in\Z_{(p)}$.
\end{remark}

\begin{remark}
    Conjecture~\ref{conj:desi2newnew} is known to be false if we remove the assumption that $G$ splits over $F^t$.  This failure was discussed for tori in Remark~\ref{rem:wildtori}.   For groups of the form $\Res_{E/F}H$ where $E/F$ is wildly ramified and $H$ is an inner form of $\mathrm{GL}_n$ over $E$, we have  $\rho(\pi) < \rho(\varphi)$ for all    positive depth irreducible representations $\pi \in \Pi(\varphi)$~\cite[Corollary~1.6]{AP22}.   When $p$ is sufficiently large, the same phenomenon happens for  $H$ being  a quasi-split classical group or a unitary group~\cite[Corollary~1.6]{AP22}.
\end{remark}

\subsection{Some evidence}\label{subsec:evidence}

We end this section by providing some evidence for Conjecture~\ref{conj:desinew}, and hence for Conjectures~\ref{conj:desi2} and~\ref{conj:desi2newnew}.

\begin{lemma}
Suppose that
\begin{enumerate}
    \item $\operatorname{char}(F)=0$.
    \item $p>2$ and is not bad \cite[Def. I.4.1]{SS70}
    for $G$.
    \item $p$ does not divide the order of $\pi_1(G_{der})$ nor divide the order of $\pi_0(Z(G))$.
    \item $G$ splits over $F^t$.
\end{enumerate}
Then Conjecture \ref{conj:desinew}
is true for Kaletha's construction of a local Langlands correspondence \cite{Kal19,Kal19b} for semisimple (a.k.a. non-singular) supercuspidal representations.
\end{lemma}

\begin{proof} Our conjectures are vacuous in the depth-zero case, so we focus on the positive depth case. A semisimple supercuspidal representation is by definition a tame supercuspidal representation constructed by Yu \cite{Yu01} with certain data specified by Kaletha through the work of Hakim-Murnaghan \cite{HM08}. The data consist of a maximal $F$-torus $T\subset G$ that is anisotropic mod $Z(G)$, a character $\theta:T(F)\ra\Cc$, and its Howe factorization $\theta=\prod_{i=-1}^d\theta_i$ for some $d\in\Z_{>0}$ \cite[\S3.6, Prop. 3.7.8]{Kal19}. Here each $\theta_i$ has depth $r_i$ with $0<r_0<r_1<...<r_{d-1}\le r_d$. 
In particular $\theta|_{T(F)_{\ge r_d}}$ is trivial and thus $\theta|_{T(F)_{\ge r_d}}:T(F)_{=r_d}\ra\Cc$ is given by $\psi_F$ and an element $X\in\ft^*(F)_{=-r_d}$, i.e. $\theta|_{T(F)_{\ge r}}=\chi_{X,F}^T$ where the latter is as in the construction at \eqref{eq:characteronlieT} with $E=F$ and $r=r_d$.

On the representation theory side, 
since $T$ is anisotropic mod $Z(G)$, 
the image of $\cB(T,F)=\{x\}$ is a single point in $\cB(G,F)$. The canonical embedding $\ft^*(F)_{=-r}\ira\fg^*(F)_{x=-r}$ gives a $1$-dimensional representation $\rho_X:G(F)_{x=r}\ra\Cc$ given by $\rho_X(W)=\psi_F(\log(W),X)$ as in \eqref{eq:MPtypeongroup}. The supercuspidal representation is then constructed as  $\pi_d=\ind_{K^d}^{G(F)}\rho_d$, the compact induction of an irreducible representation $\rho_d$ of an open subgroup $K^d\subset G(F)$ such that $K^d\supset G(F)_{x\ge r}$. By the construction in \cite[\S4]{Yu01}, the restriction $\rho_d|_{G(F)_{x\ge r}}$ is $\rho_X$-isotypic. In particular $\pi_d$ contains the  nondegenerate Moy-Prasad type $(x,X)$.

On the Galois side, while a full Langlands parameter $\varphi:W_F' \ra{}^LG$ corresponding to $\pi_d$ has a subtle construction in \cite[\S5]{Kal19} and \cite[\S4]{Kal19b}, its restriction $\varphi|_{\Group_F^{0+}}$ to the wild inertia subgroup $\Group_F^{0+}$ is rather simple: by local Langlands for tori we have $\varphi_\theta:W_F\ra{}^LT$ which can be restricted to $\varphi_\theta|_{\Group_F^{0+}}:\Group_F^{0+}\ra T^{\vee}$ and  (a much easier) part of Kaletha's construction has $\varphi|_{\Group_F^{0+}}=(T^{\vee}\ira G^{\vee})\circ\varphi_\theta|_{\Group_F^{0+}}$ up to $G^{\vee}$-conjugation. In particular, $\varphi|_{\Group_F^r}=(T^{\vee}\ira G^{\vee})\circ\varphi_\theta|_{\Group_F^r}$ up to $G^{\vee}$-conjugation. Since $\theta|_{T(F)_{\ge r}}=\chi_{X,F}^T$, we have  $\varphi_\theta|_{\Group_F^r}=\phi_{X,F}^T$.  The theorem now follows from Corollary \ref{cor:comparenew}.
\end{proof}

\section{Depth zero case} \label{sec:depthzero}
 We fix  an isomorphism
$\iota:\bar k^\times\simeq(\mathbb Q/\mathbb Z)_{p'}$,
where $(\mathbb Q/\mathbb Z)_{p'}\subset \mathbb Q/\mathbb Z$ is the subgroup of elements of 
order prime to $p$.

\subsection{Depth zero Moy-Prasad types}

 A depth zero Moy-Prasad type 
    is a pair
$( \depthzerofirstentry{x},\chi)$ where  $x\in\cB(G,F)$
    and $\chi$ is a cuspidal representation of $G(F)_{x=0}$
inflated to $G(F)_{x\geq0}$.
We denote by 
$\MP_0
$ 
the set of  
depth zero Moy-Prasad types for $G(F)$.

Suppose $S$ is a maximal $F^u$-split $F$-torus in $G$ that contains a maximal $F$-split torus of $G$.  So, for example, $S$ could be $A^u$.
Let $W_{S}=N_G(S)(F^u)/C_{G(F^u)}(S)$ be the relative Weyl group
of $G(F^u)$.
There is an algebraic $k$-torus
$\mathsf S$ such that 
$S(F^u)_{=0}=\mathsf S(\bar k)$. 
Moreover, we have a natural action of $W_S$ on $\mathsf S(\bar k)$.  This action is defined over $k$.

For any $x\in\cB(G,F)$ 
there exists a connected reductive $k$-group $\mathsf G_x$  such that
$G(F^u)_{x=0}=\mathsf G_{x}(\bar k)$. 
Assume $x$ is in the apartment $\cA_S\subset\cB(G,F)$
associated to the maximal $F$-split subtorus of $S$.  
There is a natural inclusion 
$\mathsf S\subset\mathsf G_x$,
defined over $k$,
which realizes $\mathsf S$
as a maximal $k$-torus of $\mathsf G_x$. 
Let $W_{x,S}=N_{\mathsf G_x(\bar k)}(\mathsf S(\bar k))/\mathsf S(\bar k)$ be the Weyl group of $\mathsf G_x$.
Since $N_{G(F^u)}(S) \cap G(F^u)_{x >0} = T(F^u)_{>0}$~\cite[Lemma~7.2.1]{debacker:totally}, there is a natural embedding
$i_x \colon W_{x,S}\cong N_{G(F^u)_{x\geq0}}(S)/T(\mathcal O_{F^u})\to  W_S$
that is compatible with the actions of $W_{x,S}$ and $W_S$ on $\mathsf{S}(\bar{k})$.

Let $S^\vee=X^*(S)\otimes\mathbb C^\times$ be the complex dual torus of $S$. It carries a natural action of $W_S$, and we denote by $S^\vee\quo W_S$ the GIT quotient. The Frobenius endomorphism $\changeFr:G(F^u)\to G(F^u)$
induces a morphism 
$\changeFr:S^\vee\quo W_S\to S^\vee\quo W_S$,
and we denote by $(S^\vee\quo W_S)^{\changeFr}$ its $\changeFr$-fixed points.

As noted in Section~\ref{sec:notation}, 
 for another such 
maximal $F^u$-split $F$-torus $S'$ containing a maximal  
$F$-split torus, there is a 
$g\in G(F)$ such that $\text{Ad}_g(S)=S'$,
and the conjugation action 
$\text{Ad}_g^* \colon (S')^\vee=X^*(S')\otimes\mathbb C^\times \to X^*(S)\otimes\mathbb C^\times =S^\vee$
descends to 
a canonical $\changeFr$-equivariant isomorphism
 $(S')^\vee\quo W_{S'}\cong S^\vee\quo W_S$.
We define the set of {\bf depth zero Deligne-Lusztig parameters} to be 
\begin{equation}
    \DLzero:=\lim_{S}\ (S^\vee\quo W_S)^{\changeFr}
\end{equation}
where the limit is over all the maximal $F^u$-split $F$-tori $S$ that contain a maximal $F$-split torus, as above.

\begin{lemma}\label{construction 1}
There is a map
\[\MP_0\longrightarrow\DLzero.\]

\end{lemma}
\begin{proof}
Let $(\depthzerofirstentry{x},\chi)$
be a depth zero Moy-Prasad type.
Pick a maximal $F^u$-split $F$-torus $S$ that contains a maximal $F$-split torus as above such that  
$x$ is in the apartment $\cA_{S}\subset B(G,F)$. 
The reductive quotient $\mathsf{S}$ of $S$ is a maximal $k$-torus that contains a maximal $k$-split torus in $\mathsf{G}_x$, and we  identify $X^*(S)$ with $X^*(\mathsf{S})$. Then, by the method of \cite[Section 5]{DL76}, the map (a) in \cite[Section 16]{L07} attaches to the cuspidal representation $\chi$ of $\mathsf G_x(k)$ an element $\theta_{S,\chi}\in(S^\vee\quo W_{x,S})^{\changeFr}$ that we call the Deligne-Lusztig parameter of $\chi$.
Note that the original construction 
in \cite[Section 5]{DL76} requires a choice of isomorphism $\iota:(\mathbb Q/\mathbb Z)_{p'}\cong \bar k ^\times$
and an embedding $(\mathbb Q/\mathbb Z)_{p'}\to\bar{\mathbb Q}^\times_\ell$, but as observed in \cite[Section 16]{L07} the choice of $\iota$ suffices.

Since all the apartments of $B(G,F)$ that contain $x$
are conjugate to each other by elements in $G(F)_{x\geq0}$, it follows that 
the image 
of $\theta_{S,\chi}$
along the composition of maps
\[(S^\vee\quo W_{x,S})^{\changeFr} \stackrel{i_x}\to (S^\vee\quo W_S)^{\changeFr} \to\DLzero\]
is independent of the choice of $S$ such that $x \in \cA_S$.   We let $\theta_{x,\chi}$ denote the image in $\DLzero$ of $\theta_{S,\chi}$.
The desired map is given by
\begin{equation}\label{MP to DL}    
\begin{array}{ccc}
   \MP_0=\{\text{depth zero Moy-Prasad types}\}&\to&\DLzero\\
   (\depthzerofirstentry{x},\chi)&\mapsto&\theta_{x,\chi}.  
\end{array}   \qedhere 
\end{equation}
\end{proof}

 Two depth zero Moy-Prasad types 
$(\depthzerofirstentry{x},\chi)$
and $(\depthzerofirstentry{y},\xi)$
are said to be associates provided that  there exists $g\in G(F)$ such that 
$G(F)_{x\geq0}\cap G(F)_{gy\geq0}$
surjects onto both 
$G(F)_{x=0}$ and $G(F)_{gy=0}$
and $\chi$ is isomorphic to   $\text{Ad}_g(\xi)$ (the transport of $\xi$ to $\Rep(G(F)_{gy=0})$ via $\Ad(g^{-1})^*$).
    
\begin{lemma}\label{ass}
    If two depth zero Moy-Prasad types $(\depthzerofirstentry{y},\xi),(\depthzerofirstentry{x},\chi)$ are associates, then $\theta_{y,\xi} = \theta_{x,\chi}$.
\end{lemma}

\begin{proof}
We will prove this in two steps.
We  first show that if there is a $g\in G(F)$ such that 
$(\depthzerofirstentry{y},\xi) = (\depthzerofirstentry{gx},\text{Ad}_g\chi)$, then $\theta_{y,\xi} = \theta_{x,\chi}$.   Suppose such a $g \in G(F)$ exists.  Note that $y = gx$.  
Pick a maximal $F^u$-split $F$-torus $S$ that contains a maximal $F$-split torus  such that 
$x\in\cA_S$.  Set $S'=\text{Ad}_gS$.
Then we have $y\in\cA_{S'}$,
and it follows from the construction that the parameter $\theta_{S',\xi}$
maps to the parameter $\theta_{S,\chi}$
under the transition map
$\text{Ad}_g^*:(S')^\vee\quo W_{y,S'}\cong S^\vee\quo W_{x,S}$.
Since the diagram
\[
\begin{tikzcd}
(S')^\vee\quo W_{y,S'}\ar[d,"\Ad^*_g","\sim"'{sloped, anchor=north}]\ar[r,hook,"i_{gx}"]&(S')^\vee\quo W_{S'}\ar[d,"\Ad^*_g","\sim"'{sloped, anchor=north}]\\
S^\vee\quo W_{x,S}\ar[r,hook,"i_x"]&S^\vee\quo W_{S}
\end{tikzcd}
\]
commutes, we conclude that $\theta_{x,\chi}=\theta_{y,\xi}$.

Suppose   $(\depthzerofirstentry{y},\xi)$ is associate to $(\depthzerofirstentry{x},\chi)$.  Thanks to the previous paragraph
we may assume
i) $x,y \in \cA_S$, 
ii) $G(F)_{x\geq0}\cap G(F)_{y\geq0}$
surjects onto both 
$\mathsf{G}_x(k)$ and $\mathsf{G}_y(k)$,
and iii) the pull backs of  $\chi$ and $\xi$ to  $G(F)_{x\geq0}\cap G(F)_{y\geq0}$ are isomorphic.
For $t \in \cA_S$ let $F(t)$ denote the facet in $\cA_S$ containing $t$. By arguing  as in \cite[Proposition 6.2]{MP96}, we can arrange  that (a)  $F(x)$ and $F(y)$ have a common point $z$ in their closures and (b)  the smallest affine subspace in $\cA_S$ containing $F(x)$  is equal to the smallest affine subspace in $\cA_S$ containing $F(y)$.  Thanks to (a) we have 
$\mathsf{P}_x= G(F^u)_{x\geq0}/G(F^u)_{z>0}$
and $\mathsf{P}_y=G(F^u)_{y\geq0}/G(F^u)_{z>0}$ are parabolic $k$-subgroups of $\mathsf{G}_z=G(F^u)_{z\geq0}/G(F^u)_{z>0}$
such that $\sS \leq \mathsf{P}_x \cap \mathsf{P}_y$.   Let $\sL_x \leq \sP_x$ and $\sL_y \leq \sP_y$ denote the (unique) Levi $k$-subgroups such that $\sS \leq \sL:=\sL_x\cap\sL_y$.  Since $\sL_x$ (resp. $\sL_y$) is generated by $\sS$ and the root groups corresponding to those affine roots (with respect to $F^u$) which are zero on $F(x)$ (resp. $F(y)$), from (b) we have $\sL = \sL_x = \sL_y$ and $\sL$ may be identified with both $\sG_x$  and $\sG_y$.
 Since $\chi$ is isomorphic to $\xi$ when pulled back to $\mathsf{L}$, it follows that 
$\theta_{S,\chi}$ and $\theta_{S,\xi}$ have the same image in $ S^\vee\quo W_{z,S}$, and hence in $S^\vee\quo W_S$.  We conclude that $\theta_{x,\chi} = \theta_{y,\xi}$.
\end{proof}

\subsection{Restricted depth zero parameters}
 Recall that  $\Group_F^{0+}\subset \Group_F\subset W_F$ denote the wild inertia and the inertia subgroups of the Weil group for $F$.
We have a short exact sequence
\[1\to \Group_F\to W_F\stackrel{v}\to\mathbb Z\to 1\]
We fix a geometric Frobenius  $\text{Fr}\in W_F$
such that $v(\text{Fr})=1$ and denote by $\sigma$ the  image of $\text{Fr}$ in the tame Weil group $W_F/\Group_F^{0+}$.
The subgroup $\Group_F/\Group_F^{0+}$ in the tame Weil group is called the tame inertia subgroup. 
The above short exact sequence gives rise to 
\[1\to \Group_F/\Group_F^{0+}\to W_F/\Group_F^{0+}\stackrel{v}\to\mathbb Z\to 1\]

Recall that the maximal $F$-torus $A$ was fixed in Section~\ref{sec:notation}.  Choose a Borel $F^u$-subgroup $B$ in $G$ such that $A \subset B$.
Let 
$\psi_0(G)=(X^*(A),\Delta,X_*(A),\Delta^\vee)$ be the based root datum of $G$
attached to the
Borel pair $(B,A)$ of $G$
over $F^u$. 
Let $G^\vee$ be the complex dual group of $G$.  We fix a pinning
$(G^\vee,B^\vee,A^\vee,e^\vee=\{x_\alpha\}_{\alpha\in\Delta})$,
and hence an action of $W_F$ via the homomorphism
\[\mu_G:W_F\to\text{Aut}(\psi_0(G))\cong\text{Aut}(G^\vee,B^\vee,A^\vee,e^\vee).\]
Since $G$ splits over a finite tame extension 
we
have $\Group_F^{0+}\subset\text{Ker}(\mu_G)$,
and so the action of $W_F$ 
on $G^\vee$ factors through the tame Weil group 
$W_F/\Group_F^{0+}$.

\begin{definition} \label{def:toralzeromapa} 
A \textbf{depth zero Langlands parameter}
    is a continuous cocycle 
    \[(\rho:W_F/\Group_F^{0+}\to G^\vee)\in Z^1(W_F/\Group_F^{0+},G^\vee)\]
    such that $\rho(\sigma)$ is  semi-simple.
        A \textbf{restricted depth zero parameter} is a continuous cocycle
    \[(\phi:\Group_F/\Group_F^{0+}\to G^\vee)\in Z^1(\Group_F/\Group_F^{0+},G^\vee)\]
    that admits an extension to a depth zero 
    Langlands parameter.
    We write 
    $\TPzero$ for the set of 
    $G^\vee$-conjugacy classes of 
    restricted depth zero parameters.
\end{definition}    
    
We note that $\TPzero$ is independent of our choice of pinning $(G^\vee,B^\vee,A^\vee,e^\vee)$ as all pinnings are $G^{\vee}$-conjugate.

\begin{lemma}\label{construction 2}
  There is a bijection
\[\TPzero\cong\DLzero.\]
\end{lemma}
\begin{proof}
   From~\cite[IV.2 Exc. 2b]{serre:local} and  the canonical isomorphism $\mathbb Q/\mathbb Z\cong {\dsp \lim_{{} \to {}} \Z/n \Z}$ there are canonical  
    identifications 
    \[\Group_F^{0:0+} = \Group_F/\Group_F^{0+}\cong\lim_{n>0}\mu_n(\bar k)\cong
    \Hom((\mathbb Q/\mathbb Z)_{p'},\bar k^\times).\]
    The isomorphism $\iota:\bar k^\times \cong(\mathbb Q/\mathbb Z)_{p'}$  
    defines a pro-generator $\theta_\iota\in \Group_F/\Group_F^{0+}$.
    Therefore, a continuous cocycle $\phi:\Group_F^{0:0+} \to G^\vee$ is completely determined by
    its value at $\theta_\iota$, 
   and the assignment $\phi\to \phi(\theta_\iota)\rtimes\theta_\iota$
    defines an injection 
    \[\TPzero\to (G^\vee\rtimes\theta_\iota)_{ss}/ G^\vee\]
    where the right hand side is the  set of 
    $G^\vee$-conjugacy classes of semisimple elements in $G^\vee\rtimes\theta_\iota$ (here $G^\vee$ acts on $G^\vee\cong G^\vee\rtimes\theta_\iota$ by the twisted conjugation action $
    h(g\rtimes\theta_\iota)h^{-1}=hg\theta_\iota(h)^{-1}\rtimes\theta_\iota$). 
Note that $\sigma\theta_\iota^q\sigma^{-1}=\theta_\iota$ implies 
$(G^\vee\rtimes\theta_\iota)_{ss}/ G^\vee$ is stable under the   map
    $[\sigma\circ q]:(G^\vee\rtimes\theta_\iota)_{ss}/ G^\vee\to(G^\vee\rtimes\theta_\iota)_{ss}/ G^\vee$ induced by 
 \[(G^\vee\rtimes\theta_\iota)_{ss}\to (G^\vee\rtimes\theta_\iota)_{ss}\ \ \ \ g\rtimes\theta_\iota\mapsto\sigma((g\rtimes\theta_\iota)^q)=(1\rtimes\sigma)(g\rtimes\theta_\iota)^q(1\rtimes\sigma^{-1}).\]
 We claim that  the injection above factors through
\[\TPzero\to ((G^\vee\rtimes\theta_\iota)_{ss}/ G^\vee)^{[\sigma\circ q]}\]
    where $((G^\vee\rtimes\theta_\iota)_{ss}/ G^\vee)^{[\sigma\circ q]}$ is the fixed point set of $[\sigma\circ q]$.
 Indeed, choose a depth zero parameter $\rho$ extending $\phi$, then by repeatedly using the cocycle condition we have
 \begin{equation*}
     \begin{split}
\sigma((\phi(\theta_\iota)\rtimes\theta_\iota)^q) &=
 \sigma(\phi(\theta_\iota^q)\rtimes\theta_\iota^q)=(\sigma(\phi(\theta_\iota^q))\rtimes \theta_\iota)=(\sigma(\rho(\sigma^{-1}\theta_\iota\sigma))\rtimes \theta_\iota) \\
&=\sigma(\rho(\sigma^{-1}))\rho(\theta_\iota)\theta_\iota(\rho(\sigma))\rtimes\theta_\iota=
\rho(\sigma)^{-1}\phi(\theta_\iota)\theta_\iota(\rho(\sigma))\rtimes\theta_\iota,
 \end{split}
\end{equation*}
which is in the same orbit as $ \phi(\theta_\iota)\rtimes\theta_\iota$.

To ease notation set $S = A^u$, the maximal $F^u$-split subtorus of $A$.  Then $A = C_G(S)$ and $W_S = W_A^{\theta_\iota}$ where $W_A = N_G(A)/A$.  According to \cite[Proposition 6.7]{B77} there are bijections 
    \begin{equation} \label{eq:bijectionsappb}
        (G^\vee\rtimes\theta_\iota)_{ss}/ G^\vee\cong (A^\vee\rtimes\theta_\iota)/N_{\theta_{\iota}} \cong 
    S^\vee\quo W_S
    \end{equation}
    where $N_{\theta_{\iota}}$ is the inverse image 
    of $W_S\cong W_A^{\theta_{\iota}} \subset W_A = N_G(A)/A \cong N(A^\vee)/A^\vee$ in $N(A^\vee)$. 
    Moreover, the bijections of~\eqref{eq:bijectionsappb} restrict to bijections
    \[(G^\vee\rtimes\theta_\iota)_{ss}/ G^\vee)^{[\sigma\circ q]}\cong 
    (A^\vee\rtimes\theta_\iota)/N_{\theta_\iota}) ^{[\sigma\circ q]}\cong 
    (S^\vee\quo W_S)^{[\sigma\circ q]}\cong 
    (S^\vee\quo W_S)^{\changeFr}\cong\DLzero\]
    where the bijection $(S^\vee\quo W_S)^{[\sigma\circ q]}\cong 
    (S^\vee\quo W_S)^{\changeFr}$
    comes from the equality 
    \[\chi^q=\sigma^{-1}\circ\changeFr(\chi)\]
    for $\chi\in X^*(S^\vee)\cong X_*(S)$. (Recall $\sigma^{-1} $ is the image of the arithmetic Frobenius in the tame Weil group $W_F/\Group_F^{0+}$.)
All together, we obtain an injection
\begin{equation}\label{injection}
    \TPzero\to(G^\vee\rtimes\theta_\iota)_{ss}/ G^\vee)^{[\sigma\circ q]}\cong\DLzero.
\end{equation}
It remains to show that~\eqref{injection} is surjective. 
Let $a\rtimes\theta_\iota\in A^\vee\rtimes\theta_\iota$
be a representative of an orbit 
$\gamma$ in 
$(A^\vee\rtimes\theta_\iota/ N_{\theta_\iota})^{[\sigma\circ q]}\cong(G^\vee\rtimes\theta_\iota)_{ss}/ G^\vee)^{[\sigma\circ q]}$
and let $\phi:\Group_F^{0:0+}\to G^\vee$, $\phi(\theta_\iota)=a$ be the corresponding restricted depth zero parameter.
There exists $n\in N_{\theta_\iota}$ such that 
$\sigma((a\rtimes\theta_\iota)^q)=n^{-1}(a\rtimes\theta_\iota)n
$. Since $n$ is semi-simple,  
$\phi$ admits an extension
to a depth zero Langlands parameter 
$\rho:W_F/\Group_F^{0+}\to G^\vee$ that is characterized by 
$\rho(\theta_\iota)=a$ and $\rho(\sigma)=n$ (see~\cite[8.2]{B77} for the requirements a Langlands parameter should satisfy).
It follows that~\eqref{injection} is surjective.
\end{proof}

Thanks to~\cite[Thm. 5.2]{MP94} and \cite[Thm. 3.5]{MP96} there is a map from $\Irr(G(F))_0$, the set of (equivalence classes of) irreducible depth zero representations of $G(F)$, to the set of associate classes in
    $\MP_0$.
Now combining Lemma \ref{construction 1}, Lemma \ref{ass}, and 
Lemma \ref{construction 2}, we obtain the desired map.

\begin{lemma}  \label{lem:pfinzero}
Theorem~\ref{thm:mainthm} is valid in the depth zero situation.  That is,
    there is a map 
\[
\Irr(G(F))_0 \ra \TPzero. 
\]
\end{lemma}

\begin{remark}
Similar to the positive depth case (see Remarks~\ref{rem:indofaddchar} and~\ref{rem:indofaddchar2}), the isomorphism $\iota: \bar k^\times\cong(\mathbb Q/\mathbb Z)_{p'} $
appears twice in the construction of the map in Lemma~\ref{lem:pfinzero}, namely in Lemma \ref{construction 1} and in
Lemma \ref{construction 2}.  The two appearances cancel out and
the map is 
independent of the choice of $\iota$. 
    
\end{remark}
\begin{conjecture}\label{conjectue: depth zero}
    Let $G$ 
   be a connected reductive group over $F$ that  splits over a tame extension. 
   Let $\pi$ be a depth zero 
   irreducible representation of $G(F)$ and let 
    $\phi_\pi\in\TPzero$ be its 
   restricted depth zero parameter as constructed in Lemma~\ref{lem:pfinzero}.
   Then the restriction of 
   a Langlands parameter for $\pi$ 
   to $\Group_F/\Group_F^{0+}$ lands in the $G^\vee$-conjugacy class 
   of $\phi_\pi$.
   
\end{conjecture}

    Conjecture \ref{conjectue: depth zero} is known to be true in at least the following three situations.
    (1) When $G=T$ is a tamely ramified torus the validity of Conjecture~\ref{conjectue: depth zero}  follows from the local Langlands correspondence for tori \cite{Yu09}.
    (2) In Kazhdan-Lusztig's and Lusztig's local Langlands correspondence for irreducible unipotent representations~\cite{KL87,L95} we have that Conjecture~\ref{conjectue: depth zero} is true; indeed, in this case the restriction of the Langlands parameters to $\Group_F/\Group_F^{0+}$ is trivial.
    (3)  In DeBacker-Reeder's local Langlands correspondence for depth-zero supercuspidal representations \cite{DR09}, their local Langlands correspondence is ultimately reduced, via Deligne-Lusztig induction at a reductive quotient, to the local Langlands correspondence for unramified tori. Hence, in this case the validity of Conjecture~\ref{conjectue: depth zero} follows from (1).

\subsection*{Acknowledgments} We thank Sarbartha Bhattacharya, Sean Cotner, Ram Ekstrom, Jessica Fintzen, Masao Oi, David Schwein, Loren Spice, and Jiu-Kang Yu for very helpful discussions.
T.-H. Chen also thanks the 
 NCTS-National Center for Theoretical Sciences at Taipei
 where parts of this work were
done. The research of T.-H. Chen is supported by NSF grant DMS-2143722. The research of C.-C. Tsai is supported by NSTC grant 114-2628-M-001-009-.

\def\cfgrv#1{\ifmmode\setbox7\hbox{$\accent"5E#1$}\else
  \setbox7\hbox{\accent"5E#1}\penalty 10000\relax\fi\raise 1\ht7
  \hbox{\lower1.05ex\hbox to 1\wd7{\hss\accent"12\hss}}\penalty 10000
  \hskip-1\wd7\penalty 10000\box7}
\providecommand{\bysame}{\leavevmode\hbox to3em{\hrulefill}\thinspace}
\providecommand{\MR}{\relax\ifhmode\unskip\space\fi MR }
\providecommand{\MRhref}[2]{%
  \href{http://www.ams.org/mathscinet-getitem?mr=#1}{#2}
}
\providecommand{\href}[2]{#2}

\end{document}